\documentclass{amsart}

    \usepackage[utf8]{inputenc}
    \usepackage{verbatim}
    \usepackage{amsfonts}
    \usepackage{amsmath}
    \usepackage{amsthm}
    \usepackage{amssymb} 
    \usepackage{mathrsfs} 
    \usepackage{mathtools} 
    \usepackage{stmaryrd} 
    \usepackage{bm} 
    \usepackage{enumitem} 
    \usepackage{thmtools} 
    
    \usepackage[style=alphabetic,maxbibnames=99]{biblatex}
        \DeclareFieldFormat{title}{\mkbibquote{#1}}
    
    \usepackage[hidelinks]{hyperref}
    \usepackage[noabbrev,nameinlink,capitalise]{cleveref}
        \crefname{subsection}{Subsection}{Subsections}
        \crefname{subsection}{Subsection}{Subsections}
    
    \usepackage{subcaption}
    \usepackage{color}
    \usepackage[dvipsnames]{xcolor}
    \usepackage{graphicx}\graphicspath{{./images/}}
    \usepackage{tikz}
        \usetikzlibrary{cd,arrows,arrows.meta,shapes}
        \tikzset{every picture/.style=thick}
        \tikzset{state/.style = {shape=circle,draw,inner sep=1pt, outer sep=1pt,minimum size=.5cm}}
        \tikzset{vertex/.style = {shape=circle,draw,inner sep=1pt, outer sep=1pt}}
        \tikzset{edge/.style = {->, >={latex[scale=1.5]}}}

    \theoremstyle{plain}
        \newtheorem{theorem}{Theorem}[section]
        \newtheorem{corollary}[theorem]{Corollary}
        \newtheorem{lemma}[theorem]{Lemma}
        \newtheorem{proposition}[theorem]{Proposition}
         \newtheorem{innercustomthm}{Theorem}
\newenvironment{thm}[1]
  {\renewcommand\theinnercustomthm{#1}\innercustomthm}
  {\endinnercustomthm}
  \newtheorem{innercustomthms}{Theorems}
\newenvironment{thms}[1]
  {\renewcommand\theinnercustomthms{#1}\innercustomthms}
  {\endinnercustomthms}
    \theoremstyle{definition}
        \newtheorem{definition}[theorem]{Definition}
        \newtheorem{example}[theorem]{Example}
        
    \theoremstyle{remark}
        \newtheorem{remark}[theorem]{Remark}
        
        \newtheorem{claim}[theorem]{Claim}

    \title{Limit spaces of vertex and edge replacement systems}
    \author{Davide Perego \and Matteo Tarocchi}
    \date{}
    
    \thanks{
    The first author is supported by the Swiss Government Excellence Scholarship and from Swiss NSF grant 200020-200400.
    The second author is supported by the French project GoFR (ANR-22-CE40-0004) and by a public grant from the Fondation Mathématique Jacques Hadamard.
    Both authors acknowledge support from the research grant PID2022-138719NA-I00 (Proyectos de Generación de Conocimiento 2022) financed by the Spanish Ministry of Science and Innovation.
    Both the authors are members of the Gruppo Nazionale per le Strutture Algebriche, Geometriche e le loro Applicazioni (GNSAGA) of the Istituto Nazionale di Alta Matematica (INdAM)
    }

    \address{Section de mathématiques, Université de Genève, rue du Conseil-Général 7-9, 1205 Genève, Switzerland}
    \email{\href{mailto:dperego9@gmail.com}{dperego9@gmail.com}}
    
    \address{Université Paris-Saclay, CNRS, Laboratoire de mathématiques d’Orsay, Orsay, France, EU \& Université de Rennes, CNRS, IRMAR, Rennes, France, EU}
    \email{\href{mailto:matteo.tarocchi.math@gmail.com}{matteo.tarocchi.math@gmail.com}}

    \newcommand{\col}[0]{\mathsf{c}}
    \newcommand{\type}[0]{\mathsf{t}}
    \newcommand{\Shift}[0]{\mathbb{S}}
    \newcommand{\Lang}[0]{\mathbb{L}}
    \newcommand{\Alphabet}[0]{\mathbb{A}}
    \newcommand{\dist}[0]{\mathrm{d}}
    \newcommand{\AT}[0]{\mathcal{H}}
    \newcommand{\V}[0]{\mathbf{V}}
    \newcommand{\Pred}[0]{\mathrm{P}}
    \newcommand{\bary}[1]{#1^\mathrm{b}}
    \newcommand{\R}[0]{\mathcal{R}}
    \newcommand{\Autom}[0]{\mathscr{A}}
    \newcommand{\BasilicaSpc}[0]{\mathcal{B}}
    \newcommand{\BasilicaGp}[0]{\mathrm{B}}
    \newcommand{\Grig}[0]{\mathrm{G}}
    \newcommand{\ERS}[0]{\mathcal{E}}
    \newcommand{\E}[0]{\mathsf{E}}
    \newcommand{\vtype}[0]{\mathsf{v}}
    \newcommand{\ST}[0]{\mathcal{ST}}
    \newcommand{\Crit}[0]{\mathrm{Crit}}
    \newcommand{\PCrit}[0]{\mathrm{PCrit}}
    \newcommand{\PCritColor}[4]{\begin{psmallmatrix}#1&#2\\#3&#4\end{psmallmatrix}}

\begin{document}

\begin{abstract}
We introduce and study VERSs (vertex and edge replacement systems) as a technology of graph expansions.
We consider its history graph, an augmented tree that records each graph expansion, and we provide sufficient conditions under which it is hyperbolic.
When hyperbolic, its Gromov boundary is what we call the limit space of the VERS.
We provide three examples from different areas of mathematics:
Schreier graphs and limit spaces of finitely generated contracting self-similar groups, injective post-critically finite iterated function systems and limit spaces of edge replacement systems.
\end{abstract}

    %


\maketitle


\section{Introduction}

Within the literature, plenty of recursively constructed graphs and topological spaces have been developed in multiple areas of mathematics (see for instance the three examples of this paper, briefly described later in this Introduction).
Such recursive structures are often handled on a case-by-case scenario and are not explicitly described in their generality.
The aim of this manuscript is to provide a general framework that is capable of describing many such constructions.

\subsection{Vertex and edge replacement systems}

The framework that we describe is that of what we call a \textbf{VERS}:
a \textit{vertex and edge replacement system}, which we introduce in \cref{sec:VERS}.
Essentially, a VERS consists of a starting bouquet of loops together with a finite set of rules for expanding vertices and edges according to their \textbf{type} and \textbf{color}, respectively.
When expanding, each vertex $v$ is replaced by other vertices of certain types according to the type of $v$, while each edge (say joining vertices $u$ and $v$), according to its color, is replaced by edges joining vertices generated by the expansions of $u$ and $v$.

As discussed in \cref{sec:limit:spaces}, each VERS has a \textbf{history graph}.
This is an \textbf{augmented tree}:
it has a vertical tree structure together with a horizontal structure on each level of the tree;
the $n$-th level of the history graph depicts the $n$-th expansion of the starting bouquet of loops at the root.
When the history graph of a VERS is \textbf{hyperbolic}, we define the \textbf{limit space} of the VERS as the Gromov boundary of the history graph.
We give conditions under which we say that a VERS is \textbf{expanding}, which ensure hyperbolicity:

\begin{thm}{\ref{thm:expanding:VERS:hyperbolic}}
If a VERS is expanding, then its history graph is hyperbolic.
\end{thm}

Notions of expansivity are common throughout the literature.
Essentially, the one that we provide for VERSs requires that iteratively expanding a geodesic of any possible shape (i.e., with any combination of edge colors and orientation) eventually increases its length.
Since not every geodesic necessarily appears among the VERS expansions, this is not a necessary condition for the hyperbolicity of the history graph (\cref{rmk:expanding:not:iff}).

In \cref{sec:self:similar:groups,,sec:IFS,,sec:ERSs} we will use VERSs to describe recursively generated graphs and their limit spaces in three main families of examples, each belonging to a distinct area of mathematics.

\subsection{Schreier graphs and limit spaces of self-similar groups (\cref{sec:self:similar:groups})}

Contracting self-similar groups (\cite[Section 2.11]{NekSSG}) make up a well-known class of groups that provides various examples of Burnside groups and groups of intermediate growth, together with deep connections with automata theory, dynamical systems \cite{GLN}, spectral graph theory \cite{GLN2, BGJRT} and other areas.
As they act on rooted trees, there is a natural interest towards the finite Schreier graphs for their actions on each level \cite{Bondarenko,BDN}.
Related to such graphs, limit spaces of contracting self-similar groups were defined by Nekrashevych in \cite{LimitSets}.

We will build a VERS $\R(G,S)$ associated to any contracting self-similar group $G$ generated by some finite $S \subseteq G$.

\begin{thm}{\ref{thm:self:similar:limit:sets}}
The history graph of $\R(G,S)$ is the augmented tree described in \cite{LimitSets}.
In particular, it is hyperbolic and the limit space of the VERS is the limit space of $G$.
\end{thm}

In particular, these VERSs provide a unified recursive procedure to build Schreier graphs, inspired by already existing constructions for tile graphs in the bounded automata case \cite{BDN}.

\subsection{Attractors of iterated function systems (\cref{sec:IFS})}

Attractors of iterated function systems provide a classic source of fractals \cite[Chapter 9]{Falcon}.
An iterated function system is a finite collection of contracting maps and its attractor is the invariant space.
In the last two decades, the literature witnessed multiple hyperbolic graphs linked to iterated function systems, usually with the aim of identifying their attractor with the Gromov boundaries of the graphs \cite{Kaimanovich,KLLW}.
This identification is often of a metric nature metric, as opposed to purely topological.

We build a VERS $\R_\Phi$ related to any injective post-critically finite iterated function system $\Phi$.

\begin{thms}{\ref{thm:IFS:history:graph} and \ref{thm:IFS:hyperbolic}}
Let $\Phi$ be an injective post-critically finite IFS.
Then the history graph of $\R_\Phi$ is hyperbolic and its horizontal edges correspond to intersections of same-length iterations of maps on the attractor.
Moreover, the limit space of $\R_\Phi$ is homeomorphic to the attractor of $\Phi$.
\end{thms}

The limit spaces of these VERSs are homeomorphic to the attractors of the iterated function systems.
Similar constructions (such as those in \cite{KLLW,,SSIFS}) produce history graphs whose levels are arranged according to the contracting ratios, so that the Gromov boundary is in some way ``metrically equivalent'' to the attractor.
VERSs provide a more general construction (for instance, we do not require the maps to be similarities, as opposed to \cite{SSIFS}), at the cost of only focusing on the topology of the attractor.
The hyperbolicity of both history graphs ultimately relies on similar arguments.

\subsection{Expansions and limit spaces of edge replacement systems (\cref{sec:ERSs})}

Edge replacement systems were developed in \cite{BF19} in order to codify a plethora of well-known fractals and to introduce their \textit{rearrangement groups}, which are countable groups of ``piecewise-rigid'' homeomorphisms of the fractal in the same vein as Thompson's groups $F$, $T$ and $V$.
Apart from an interest in their algebraic, geometric and algorithmic properties, rearrangement groups found applications in the study of larger groups acting on fractals \cite{NeretinBasilica,,BasilicaQuasiSymmetries,,QuasiSymmetries}, and edge replacement systems alone were the subject of \cite{RatGlue}.

Differently from VERSs, in edge replacement systems new vertices are added when expanding edges.
To address this difference, we build a VERS $\R_\ERS$ whose expansions are the barycentric subdivisions of the expansions of a given edge replacement system $\ERS$ (\cref{cor:ERS:to:VERS:full:expansion}).

\begin{thms}{\ref{thm:expanding:ERS:is:expanding:VERS} and \ref{thm:ERS:limit:space:is:VERS:limit:space}}
Given an edge replacement system $\ERS$ that is expanding (in the sense of \cite{BF19}), the VERS $\R_\Phi$ is expanding (in the sense of \cref{def:expanding:VERS}), so its history graph is hyperbolic.
Moreover, the limit space of $\R_\ERS$ is homeomorphic to the limit space of $\ERS$.
\end{thms}

\subsection{Final remarks}

Apart from the three examples above, we believe that VERSs are a suitable formalism for describing other recursive graph constructions.
For example, it is likely that the recursive structure hinted at in \cite[Example 3.2]{Perego} regarding atoms of hyperbolic groups (see \cite{BBM}) forms a VERS.

Observe also that limit spaces of VERSs need not be finitely ramified (in the sense of \cite{T08}).
There are examples coming from \cref{sec:self:similar:groups}:
the limit space of a finitely generated contracting not bounded self-similar group $G$ is not finitely ramified (\cite[Corollary IV.23]{Bondarenko}).
Other examples arise from \cref{sec:ERSs}:
the limit spaces of certain edge replacement systems are not finitely ramified.
On the other hand, all of the limit spaces that arise from \cref{sec:IFS} are finitely ramified.


\section{Vertex and edge replacement systems}
\label{sec:VERS}

In this section we introduce the titular technology, right after revising some basic notions and setting some useful notation.

\subsection{Background and notation}

Let us first fix some standard notions and the related notation.

\subsubsection{Graphs, colors and types}
\label{ssub:graphs}

A \textbf{graph} $\Gamma$ consists of a set $V(\Gamma)$ of \textit{vertices}, a set $E(\Gamma)$ of \textit{edges} and two maps $\iota, \tau \colon E(\Sigma) \to V(\Sigma)$ sending each edge to its \textit{initial} and \textit{terminal vertex}.
Note that parallel edges (i.e., $e_1$ and $e_2$ such that $\iota(e_1)=\iota(e_2)$ and $\tau(e_1)=\tau(e_2)$) and loops (i.e., $e$ such that $\iota(e)=\tau(e)$) are allowed.

We will say that two edges $e_1$ and $e_2$ are \textbf{adjacent} if they share a vertex (i.e., $\{\iota(e_1),\tau(e_1)\}\cap\{\iota(e_2),\tau(e_2)\}\neq\emptyset$) and that an edge $e$ is incident on a vertex $v$ if $v \in \{\iota(e),\tau(e)\}$.
For the sake of brevity, we will write that $e$ is $u \rightarrow v$ when $u=\iota(e)$ and $v=\tau(e)$.
Note that there may be multiple edges $u \rightarrow v$, since parallel edges are allowed.
We will also write that $e$ is $u \leftrightarrow v$ if $\{\iota(e),\tau(e)\} = \{u,v\}$.

A \textbf{walk} is a sequence of edges $e_1, e_2, \dots, e_k$ (possibly $k=\infty$) such that $e_i$ and $e_{i+1}$ are adjacent for all $i=1,2,\dots,k-1$.
A \textbf{path} is a walk without backtracking, i.e., such that $e_i \neq e_{i+1}$ for all $i=1,2,\dots,k-1$.
A \textbf{cycle} is a closed path, i.e., such that the first and last vertices are the same.
Note that walks and paths do not take edge orientation into consideration.
We regard $V(\Gamma)$ as a metric space with the distance between two vertices given by the shortest path joining them.
Such distance is infinite when the two vertices lie in distinct connected components.

A graph $\Gamma$ has \textbf{colors} $C$ and/or \textbf{types} $T$ if it comes equipped with maps $\col \colon E(\Gamma) \to C$ and/or $\type \colon V(\Gamma) \to T$.
Throughout the paper, $C$ will be a fixed set of colors and $T$ will be the set of vertices of a graph $\Sigma$.
The following notion allows us to have consistency between colors and types.

\begin{definition}
\label{def:kappa:compatible:colors:types}
Let $\Gamma$ be a graph that has colors $C$ and types $T$ and let $\kappa$ be a map $C \to T \times T,\ c \mapsto (\kappa(c)_1, \kappa(c)_2)$.
We say that $\Gamma$ is \textbf{$\kappa$-compatible} if, for all $e \in E(\Gamma)$, one has that $\type(\iota(e))=\kappa(c)_1$ and $\type(\tau(e))=\kappa(c)_2$.
\end{definition}

\subsubsection{Gromov boundaries of hyperbolic graphs}

Our notion of hyperbolic graphs will be that of Gromov, see for example \cite[Chapter III.H]{BridsonHaefliger}.
The hyperbolic graphs of this paper will all be augmented trees, which we will define in \cref{sec:limit:spaces}.
Such graphs have a handy characterization of hyperbolicity given by Kaimanovich in \cite{Kaimanovich}, which we will recall in \cref{thm:Kaimano}.
For this reason, we will never truly need to use the direct definition of a hyperbolic graph and we will fully rely on the ``no big squares'' condition of \cref{thm:Kaimano} by Kaimanovich.\

Let us recall the basics about Gromov boundaries, see for example \cite[Chapter III.H, \S3]{BridsonHaefliger}.
Fix a hyperbolic graph $\Gamma$ and an arbitrary vertex $v_0 \in V(\Gamma)$.

\begin{definition}
\label{def:asymptotic:relation}
Let $a_0, a_1, \dots$ and $b_0, b_1, \dots$ be the vertices of two infinite geodesic paths starting from $v_0$ ($a_0=b_0=v_0$).
We say that they are \textbf{asymptotically equivalent} if there exists some $D \geq 0$ such that $\forall i \in \mathbb{N},\ d_\Gamma(a_i,b_i) \leq D$.
\end{definition}

This defines an equivalence relation, which we will refer to as the \textbf{asymptotic relation} of $\Gamma$.
It is known that it does not depend on the choice of base vertex $v_0$, so the following definition makes sense and only depends on $\Gamma$.

\begin{definition}
\label{def:Gromov:boundary}
The \textbf{Gromov boundary} $\partial\Gamma$ of a hyperbolic graph $\Gamma$ is the quotient of the set of paths starting from $v_0$ under the asymptotic relation.
\end{definition}

It is standard to endow the Gromov boundary with the so-called visual topology, which makes it a compact metrizable topological space.

\subsubsection{Edge shifts}

A \textbf{directed walk} is a walk such that $\tau(e_i)=\iota(e_{i+1})$ (i.e., following the edge orientation).
A \textbf{directed path} is a directed walk that is a path.

If $\Sigma$ is a finite graph and $s \in V(\Sigma)$, we denote by $\Shift(\Sigma,s)$ the \textbf{initial edge shift} on $\Sigma$ starting at $s$, which is the set of all infinite directed paths on $\Sigma$ that start from $s$.
The \textbf{alphabet} of an edge shift $\Shift(\Sigma,s)$ is the set of edges of $\Sigma$ and the \textbf{language} is the set $\Lang\Shift(\Sigma,s)$ of all finite directed paths on $\Sigma$ that start from $s$.

The elements of $\Shift(\Sigma,s)$ (respectively, $\Lang\Shift(\Sigma,s)$) correspond to left-infinite words $\dots e_3 e_2 e_1$ (respectively, finite words $e_k \dots e_2 e_1$) in the alphabet $E(\Gamma)$, where $\iota(e_1)=s$ and $\tau(e_i)=\iota(e_{i+1})$ for all $i \geq 1$.
Note the unusual left direction of words, which will be useful in \cref{sec:self:similar:groups,,sec:IFS}.
It is useful to have the language $\Lang\Shift$ of an edge shift $\Shift$ include the empty path, corresponding to an empty word denoted by $\varepsilon$.

The \textbf{type} of a path $w \in \Lang\Shift(\Sigma,s)$ is the terminal vertex of the path $w$ in $\Sigma$.
The type of $\varepsilon$ is set as the initial vertex $s$ of the edge shift $\Shift(\Sigma,s)$.
Note that the type determines which digits are allowed next, i.e., for all $w_1, w_2 \in \Lang\Shift(\Sigma,s)$, if $x w_1 \in \Lang\Shift(\Sigma,s)$ then $x w_2 \in \Lang\Shift(\Sigma,s)$ if and only if $\type(w_1) = \type(w_2)$.

In the context of this paper, a fixed graph $\Sigma$ will determine an edge shift $\V=\Shift(\Sigma,s)$ and other graphs $\Gamma_n$ will have elements of $\Lang\V$ as vertices.
Thus, the type of a path $w \in \Lang\V$ on $\Sigma$ will actually determine a typing of the vertices of graphs $\Gamma_n$, in accordance with our previous description of the term \textit{type} in \cref{ssub:graphs}.

If $\Alphabet$ is a finite set, we will denote by $\Alphabet^n$ the set of words of a fixed length $n$, by $\Alphabet^*$ the set of all finite words and by $\Alphabet^{-\omega}$ the set of all left-infinite words.
If $w \in \Alphabet^*$, we will denote by $w^{-\omega}$ the periodic word $\dots w w$.

\subsection{Definition of VERSs and their limit spaces}

Let us now introduce the main technology of the paper. 

\begin{definition}
\label{def:VERS}
A \textbf{vertex and edge replacement system} (\textbf{VERS} for short) is the collection of the following data.
\begin{itemize}
    \item An initial edge shift $\V = \Shift(\Sigma,s)$.
    \item A finite set $C$ of colors.
    \item A map $\kappa \colon C \to V(\Sigma) \times V(\Sigma)$.
    \item For each color $c \in C$, a graph $R_c$ colored by $C$, typed by $\Lang\V$ and $\kappa$-compatible, whose set of vertices is
    $$\{ x\mathrm{i}, y\mathrm{t} \mid x,y \in E(\Sigma) \text{ such that } \iota(x)=\kappa(c)_1 \text{ and } \iota(y)=\kappa(c)_2 \}$$
    with $\type(x\mathrm{i})=\type(x\mathrm{t})=\tau(x)$ (here $\mathrm{i}$ and $\mathrm{t}$ are just symbols).
    \item A graph $\Gamma_0$ colored by $C$, typed by $\Lang\V$ and $\kappa$-compatible, consisting just of a bouquet of loops at the vertex $\varepsilon \in \Lang\V$.
\end{itemize}
\end{definition}

Many examples are given later on, see \cref{sub:self:similar:examples,,ex:IFS:Sierpiński:triangle,,ex:ERS:basilica}.

Note that, when the edge shift $\V$ is a full shift (i.e., when $V(\Sigma) = \{s\}$ and so $\V = E(\Sigma)^{-\omega}$), the map $\kappa$ is necessarily $c \mapsto (s,s)$ and thus can safely be omitted.

A VERS defines a sequence of graphs $\Gamma_n$ as follows.
Each graph $\Gamma_n$ has vertices $\Lang_n\V$, the set of finite words of length $n$ in the language of the initial edge shift $\V$.
From the $n$-th graph $\Gamma_n$, define $\Gamma_{n+1}$ as the graph whose vertices are $\Lang_{n+1}\V$ and whose edges are those obtained by expanding all of the edges of $\Gamma_n$.
By \textit{expanding} here we mean the following operation:
if $u \rightarrow v$ is an edge of $\Gamma_n$, then $\Gamma_{n+1}$ has edges obtained from those of $R_c$ by replacing $\mathrm{i}$ with $u$ and $\mathrm{t}$ with $v$ (so the vertices involved all have the form $xu$ and $xv$).

More generally, any $\kappa$-compatible graph $\Gamma$ can be expanded:
\begin{itemize}
    \item the vertices are $xu$, spanning all vertices $u$ of $\Gamma$ and all $x \in E(\Sigma)$ such that $\iota(x)=\type(u)$, with type $\type(ux) = \tau(x)$;
    \item the edges are obtained by replacing each edge $e$ of $\Gamma$ with a copy of $R_{\col(e)}$;
    if $e$ is $u \rightarrow v$, then the edge $x \mathrm{a} \to y \mathrm{a}$ of $R_{\col(e)}$ ($\mathrm{a} \in \{\mathrm{i},\mathrm{t}\}$) becomes $xu \rightarrow yv$.
\end{itemize}

If $\Lambda$ is a subgraph of a $\kappa$-compatible graph $\Gamma$, then clearly $\Lambda$ is $\kappa$-compatible and its expansion is a subgraph of the expansion of $\Gamma$.
Because of this, even if the VERS expansions are the graphs $\Gamma_n$, it will often make sense to consider the expansion of arbitrary graphs $\Gamma$, generally with the idea that they may appear as subgraphs of some $\Gamma_n$.

\begin{proposition}
\label{prop:expansion:kappa:compatible}
If $\Gamma$ is a $\kappa$-compatible graph, then its expansion is $\kappa$-compatible.
In particular, every $\Gamma_n$ is $\kappa$-compatible.
\end{proposition}

\begin{proof}
Assume that $\Gamma$ is $\kappa$-compatible and let $\Gamma'$ be its expansion.
Then $\Gamma'$ has edges that are colored according to the $\kappa$-compatible graphs $R_c$, so $\Gamma'$ is $\kappa$-compatible.
\end{proof}


\section{Limit spaces of expanding VERSs}
\label{sec:limit:spaces}

In this section we define and discuss the history graph of VERSs, their limit spaces (when the history graph is hyperbolic) and a condition for hyperbolicity.

\begin{definition}
\label{def:history:graph}
Given a VERS $\R = \left( \V=\Shift(\Sigma,s), C, \kappa, (R_c)_{c \in C}, \Gamma_0 \right)$, its \textbf{history graph} is the graph $\AT_\R$ (or simply $\AT$ if $\R$ is understood) defined as follows.
\begin{itemize}
    \item The set of vertices of $\AT$ is $\Lang\V$.
    \item If $v, xv \in \Lang\V$ for some $x \in E(\Sigma)$, then $\AT$ has a \textbf{vertical edge} $v \rightarrow xv$.
    \item If $u \rightarrow v$ is an edge of some $\Gamma_n$, then it is a \textbf{horizontal edge} of $\AT$.
\end{itemize}
\end{definition}

\begin{definition}
\label{def:VERS:limit:space}
The \textbf{limit space} of a VERS $\R$ is the Gromov boundary of the history graph $\AT_\R$, provided $\AT_\R$ is hyperbolic.
\end{definition}

As we discuss in the remainder of this section, history graphs are hyperbolic for an interesting class of VERSs (\cref{thm:expanding:VERS:hyperbolic}).

\subsection{The history graph of a VERS is an augmented tree}

Given a tree $T$ with root $r$, let us denote by $\Pred \colon V(T) \setminus \{r\} \to V(T)$ the map that sends any vertex $u$ to its immediate predecessor.

\begin{definition}
\label{def:augmented:tree}
An \textbf{augmented tree} is a graph $T$ equipped with a bipartition of the set of edges into a set of \textit{vertical edges} and one of \textit{horizontal edges} satisfying the following properties:
\begin{itemize}
    \item the subgraph of $T$ consisting of the vertical edges is a rooted tree;
    \item if $u \rightarrow v$ is a horizontal edge, then there is a horizontal edge between the immediate predecessors $\Pred(u)$ and $\Pred(v)$.
\end{itemize}
\end{definition}

Note that this is not precisely the standard definition (e.g., Definition 3.2 of \cite{Kaimanovich}).
The sole difference is that we allow for loops and parallel edges, which does not alter the metric structure of the graph.

\begin{definition}
\label{def:lift}
In an augmented tree, the \textbf{lift} of a horizontal edge $u \leftrightarrow v$ is any edge $\Pred(u) \leftrightarrow \Pred(v)$ (which is necessarily horizontal).
\end{definition}

\begin{remark}
\label{rmk:lift:distance}
In any augmented tree, if $P = e_1, \dots, e_k$ is a horizontal path between vertices $u$ and $v$, the lifts of $e_1, \dots, e_k$ form a walk between $\Pred(u)$ and $\Pred(v)$.
In particular, there is a path of length at most $k$ between $\Pred(u)$ and $\Pred(v)$, so their horizontal distance is at most $k$.
In simple terms, this means that lifting cannot increase distances, in an augmented tree.
\end{remark}

\begin{proposition}
The history graph $\AT$ of any VERS is an augmented tree.
\end{proposition}

\begin{proof} 
The graph $\AT$ comes equipped with a bipartition into horizontal and vertical edges.
Let us see that it satisfies the two conditions of \cref{def:augmented:tree}.

The subgraph of $\AT$ that only consists of the vertical edges is the graph of the language of the initial edge shift $\V$, which is a rooted tree.

Assume that $u \rightarrow v$ is a horizontal edge of $\AT$ ($u$ and $v$ may be the same), say that it is an edge of $\Gamma_n$.
Let us consider the predecessors $\Pred(u)$ and $\Pred(v)$, which are vertices of $\Gamma_{n-1}$ (possibly the same).
By the very definition of the graphs $\Gamma_m$, we know that $e' = \Pred(u) \leftrightarrow \Pred(v)$ is an edge of $\Gamma_{n-1}$.
This means that there is a horizontal edge $e'$ in $\AT$ between the predecessors of $u$ and $v$, as needed.
\end{proof}

Every edge of an augmented tree has a lift (by the very definition of augmented tree, \cref{def:augmented:tree}), but in general there may be multiple lifts of the same edge.
The augmented tree $\AT_\R$ of a VERS $\R$ has additional structure:
for each edge, among its lifts there is a ``canonical'' one:

\begin{definition}
\label{def:spanning:lift}
Given a horizontal edge $e = xv \rightarrow yw$ ($v,w,xv,yx \in \Lang\V$ and $x,y \in \Alphabet\V$) in the augmented tree of a VERS, say that it belongs to $\Gamma_n$, its \textbf{spanning lift} is the unique edge $v \leftrightarrow w$ of $\Gamma_{n-1}$ (and so a horizontal edge in the augmented tree) whose expansion produces $e$.
\end{definition}

\subsection{Expanding VERS have hyperbolic history graphs}

Various notions of expansivity have been developed in multiple areas such as group theory, automata, graphs and dynamical systems (e.g. \cite{BF19,KLLW,NekSSG,HP}).
Here we provide one for our replacement systems.

\begin{definition}
\label{def:expanding:VERS}
A VERS is $n$-\textbf{expanding} if, given any path $\Gamma$ of length $n$, its $n$-th expansion does not feature any geodesic of length $n$ joining vertices successors of the endpoints of $\Gamma$.
The VERS is \textbf{expanding} if there exists an $n \in \mathbb{N}$ such that it is $n$-expanding.
\end{definition}

The idea behind such notion is to make sure that any two vertices of a graph see the distances of their successors grow larger, eventually (after some finite amount of expansions), whatever the shape of the geodesics between them.

\begin{remark}
Whether a VERS is $n$-expanding or not for a fixed $n \in \mathbb{N}_{>0}$ is decidable:
one simply needs to consider the $n$-th expansions of all possible paths of length $n$ (i.e., for all possible choices of colors and orientations of the edges of such path) and check whether they feature geodesics of length $n$ joining the successors of the endpoints of the original path.
This provides a naive algorithm which shows that the problem of determining whether a VERS is expanding is semi-decidable.
\end{remark}

Let us recall the main tool (from \cite{Kaimanovich}) that we will use to show that every expanding VERS has a hyperbolic history graph.
A \textbf{geodesic $n$-square} is a cycle of an augmented tree $\AT$ that decomposes as four consecutive geodesics $H_1$, $V_1$, $H_2$, $V_2$ of $\AT$, each of length $n$, of which $H_1$ and $H_2$ are fully comprised of horizontal edges and $V_1$ and $V_2$ of vertical edges.

\begin{theorem}[Definition 3.12 and Theorem 3.13 of \cite{Kaimanovich}]
\label{thm:Kaimano}
An augmented tree is hyperbolic if and only if it satisfies the \textbf{``no big squares'' condition}, i.e., the size of geodesic squares is bounded.
\end{theorem}

Note that the following lemma is about any augmented tree.

\begin{lemma}
\label{lem:squares:lift}
If an augmented tree has geodesic $(n+1)$-squares, then it has geodesic $n$-squares.
\end{lemma}

\begin{proof}
Consider a square $(T,L,B,R)$ of size $n+1$ where $T$ and $B$ are the horizontal geodesics, $T$ the nearest to the root.
Let $u$ and $v$ be the endpoints of $B$ and let $u'=\Pred^n(u)$ and $v'=\Pred^n(v)$ be the endpoints of $T$.
We will essentially prove that the square $(T,L,B,R)$ has an underlying ``grid'' of subsquares.

By \cref{rmk:lift:distance}, lifting the bottom geodesic $B_0 \coloneq B$ produces a walk $B_1$ of length at most $n+1$.
Iteratively, we can lift $B_i$ to $B_{i+1}$, obtaining a sequence of walks $B_0=B, B_1, \dots, B_n$, each of length at most $n+1$.
We claim that each $B_i$ is actually a geodesic of length $n+1$.
Let us start by proving that the length of each $B_i$ is $n+1$, so assume by contradiction that some $B_i$ has length $l \leq n$.
Then, by the aforementioned \cref{rmk:lift:distance}, for each $j = i+1, \dots, n$ the walk $B_j$ has length at most $l$.
In particular, this implies that $B_n$ is a walk of length strictly less than $n+1$.
Now, $B_n$ is a walk that joins $u'$ and $v'$, so their distance is strictly less than $n+1$.
This contradicts the fact that $T$ is a geodesic of length $n+1$ between $u'$ and $v'$, which must hold by the assumption the $(T,L,B,R)$ is a geodesic square of size $n+1$.
This shows that each $B_i$ has length $n+1$.
Each of these now must be a geodesic by essentially the same argument:
if some $B_i$ were not a geodesic, any geodesic $P$ between the endpoints of $B_i$ would have length strictly less than $n+1$, so lifting it would produce a geodesic between $u'$ and $v'$ of length strictly less than $n+1$, once again contradicting the fact that $T$ is a geodesic of length $n+1$.

Now let $B_i'$ be the subgeodesic of $B_i$ obtained by removing from each $B_i$ the vertex that it shares with $R$.
Let us write $B'$ for $B_0'$ and $T'$ for $B_{n-1}'$ (note that $B$ and $B'$ have the same height, whereas $T'$ and $T$ do not).
Let $L'$ be the subgeodesic of $L$ obtained by removing from $L$ the vertex that it shares with $T$.
We claim that there is a geodesic square $(T',L',B',R')$ of size $n$, so let us see that there exists a geodesic $R'$ of length $n$ between the terminal endpoints of $B'$ and $T'$.
Consider the last edge $w \rightarrow v$ of $B$ (the closest to $R$).
This lifts to some edge $\Pred(w) \rightarrow \Pred(v)$ that belongs to $B_1$.
Since $\Pred(w)$ is the immediate predecessor of $w$, there is a vertical edge $\Pred(w) \rightarrow w$.
Iterating this argument, we found a vertical path $\Pred^{n-1}(w) \rightarrow \Pred^{n-2}(w) \rightarrow \dots \rightarrow w$ which we call $R'$.
Then $(T',L',B',R')$ is a geodesic square of size $n$, as needed.
\end{proof}

\begin{theorem}
\label{thm:expanding:VERS:hyperbolic}
If a VERS is expanding, then its history graph $\AT$ is hyperbolic.
\end{theorem}

\begin{proof}
Assume that the given VERS is $n$-expanding.
We claim that its history graph $\AT$ has no squares of size $n$.
Observe that, if $\AT$ has no squares of size $n$, then it has no squares of size $n+1$ by \cref{lem:squares:lift}.
Thus, by \cref{thm:Kaimano}, our claim implies that $\AT$ is hyperbolic.

Assume by contradiction that $\AT$ has a square $(T,L,B,R)$ of size $n$.
If $B$ is $b_0 \leftrightarrow b_1 \leftrightarrow \dots \leftrightarrow b_n$, then consider the vertices $t_i \coloneq \Pred^n(b_i)$ for $i=0,1,\dots,n$.
If $e_i = b_{i-1} \leftrightarrow b_i$ are the edges of $B$, consider their spanning lifts $s_i = t_{i-1} \leftrightarrow t_i$.
Note that $T$ is not necessarily comprised of the edges $s_i$.

Observe that the $t_i$'s are all distinct, otherwise there would be a path between $t_0$ and $t_n$ that is shorter than $T$, so $T$ would not be a geodesic.
Then the edges $s_i$ are all distinct too, meaning that they form a path of length $n$ between $t_0$ and $t_n$.
This means that there is a path $s_1, \dots, s_n$ of length $n$ joining $t_0$ and $t_n$ and whose $n$-th expansion features the geodesic $B$, which is of length $n$.
This contradicts \cref{def:expanding:VERS}, concluding the proof.
\end{proof}

We will use this result in \cref{sec:ERSs} in order to prove \cref{thm:ERS:limit:space:is:VERS:limit:space}.

\begin{remark}
\label{rmk:expanding:not:iff}
If $\AT$ is hyperbolic, the VERS may not be expanding.
Indeed, paths that cause the expanding condition to fail may not occur in expansions of the base graph.
The VERS depicted in \cref{fig:grigorchuk:group:VERS}, which we will introduce in \cref{sub:self:similar:examples}, serves as an example:
a path comprised of $n$ edges whose colors are among $b$, $c$ and $d$ expands to a path with the same property, so its $n$-th expansion produces a geodesic $n$-square and thus the VERS is not expanding.
This is not an obstruction to hyperbolicity simply because expansions of the base graph never feature such paths.
\end{remark}


\section{Schreier graphs and limit spaces of finitely generated contracting self-similar groups}
\label{sec:self:similar:groups}

In this section we build VERSs whose history graphs are the finite Schreier graphs of finitely generated contracting self-similar groups and whose boundaries are the limit spaces of such groups.
Essentially, we show that the notions introduced in \cite{NekSSG} fit nicely into the framework of VERSs.

\subsection{Background on self-similar groups and Nekrashevych's self-similarity graphs}

Throughout this section, let $\Alphabet$ denote the set $\{0,1,\dots,d-1\}$, for some fixed $d \in \mathbb{N}_{\geq1}$ and let $\Lang = \Alphabet^*$.
This is consistent with our previous use of the symbols $\Alphabet$ and $\Lang$, since $\V$ will be a full shift throughout the section.

\subsubsection{Contracting self-similar groups}

Let us briefly recall the definitions and notation of self-similar groups.
See \cite{NekSSG} for much more detail.

A \textbf{self-similar group} is a group $G$ of permutations of $\Alphabet^*$ such that
$$\forall g \in G,\ x \in \Alphabet,\ \exists h \in G,\ y \in \Alphabet \mid \forall w \in \Alphabet^*,\ (xw)g = y (w)h.$$
The element $h$ will be denoted by $g_x$ (the \textit{restriction} of $g$ to $x$) and $\sigma_g$ will denote the permutation of $\Alphabet$ that maps $x$ to $y$.
Together, $\sigma_g$ and $\{g_i\}_{i\in\Alphabet}$ uniquely determine the element $g$.
See \cite[Section 1.5]{NekSSG} for more detail and explanations.

The notion of restriction extends to finite words with an inductive definition:
if $w \in \Alphabet^*$ and $x \in \Alphabet$, then $g_{xw}$ will denote $(g_x)_w$, i.e., the restriction to $w$ of the restriction of $g$ to $x$.
A self-similar group $G$ is \textbf{contracting} if there exists a finite $N \subseteq G$ such that, for all $g \in G$, there exists $k \in \mathbb{N}$ such that $g_w \in N$ for all words $w$ of length at least $k$.

\subsubsection{Self-similarity graphs and limit spaces}

The notions that we briefly introduce here are from \cite{LimitSets}, see also \cite[Section 3.7]{NekSSG}. 

Given a self-similar group $G$ generated by a finite set $S$, the \textbf{self-similarity graph} $\Sigma(G,S)$ has vertices $\Lang$, vertical edges $w \rightarrow xw$ for all $w \in \Lang$ and $x \in \Alphabet$ and horizontal edges $u \xrightarrow{s} v$ for all $u,v \in \Lang$ and $s \in S$.
This is an augmented tree as soon as the restrictions of all elements of $S$ belong to $S$.
Note that the connected components of the graph spanned by horizontal edges of $\Sigma(G,S)$ are exactly the finite Schreier graphs of the action of $G$ on the sets of words of fixed length.

As shown in \cite{LimitSets}, if $G$ is contracting then $\Sigma(G,S)$ is hyperbolic and different choices of $S$ produce quasi-isometric self-similarity graphs, so the Gromov boundary of $\Sigma(G,S)$ does not depend on the choice of $S$.
Moreover, the Gromov boundary of $\Sigma(G,S)$ is homeomorphic to the limit space of $G$, defined in \cite[Definition 3.3]{NekSSG} as a quotient of $\V$.

\subsection{VERSs for \texorpdfstring{$\Sigma(G,S)$}{Sigma(G,S)} of contracting self-similar groups}
\label{sub:self:similar:VERS}

Let us show how to build VERS for finitely generated contracting self-similar groups.

\begin{remark}
It is known that any finitely generated contracting self-similar group $G$ admits a finite generating set $S$ that is closed under taking restrictions, i.e., $\forall g \in S,\ \forall x \in \Alphabet,\ g_x \in S$.
Indeed, given any finite generating set $S'$, one can construct a finite generating set $S$ that is closed under restrictions by adding to $S'$ all restrictions of its elements, which are finitely many by the contracting property.
\end{remark}

Let us define a VERS $\R(G,S)$ based on $G$ and a finite generating set $S$ that is closed under restrictions.
\begin{itemize}
    \item $\V$ is a full shift, i.e., $\V=\Shift(\Sigma,t)$ where $V(\Sigma) = \{t\}$.
    \item The set of colors is $S$.
    \item The map $\kappa$ sends each $g \in S$ to $(t,t)$.
    \item For each $g \in S$, the replacement graph $R_g$ has, for all $x,y\in\Alphabet$ and $h \in S$ such that $\sigma_h(x)=y$ and $h_x=g$, an $h$-colored edge from $x\mathrm{i}$ to $y\mathrm{t}$.
    \item The base graph $\Gamma_0$ consists of one $g$-colored loop for each $g \in S$.
\end{itemize}
This defines a VERS that we denote by $\R(G,S)$.

\begin{lemma}
\label{lem:self:similar:parallel:edges}
The expansions of any $\R(G,S)$ never feature parallel edges with the same color.
\end{lemma}

\begin{proof}
Assume by contradiction that some $\Gamma_n$ has two edges $e_1$ and $e_2$ from $xu$ to $yv$ ($x,y \in \Alphabet$, $u,v \in \Lang$) with the same color $g$.
Proceeding by induction on $n$, let us assume that $\Gamma_i$ has no parallel edges of the same colors for any $i \leq n-1$, which is true for the base graph $\Gamma_0$.
Let $f_1$ and $f_2$ be the spanning lifts of $e_1$ and $e_2$, respectively (\cref{def:spanning:lift}).
By the definition of $\R(G,S)$, no replacement graph features parallel edges of the same color, so $f_1$ and $f_2$ must be different edges.
They are both edges from $u$ to $v$, so they cannot have the same color by induction hypothesis, say that $g_1 \neq g_2$ are their respective colors.
By definition of $\R(G,S)$, for each $i \in \{1,2\}$, since $e_i$ is an edge of $\Gamma_n$ we have that $g_x=g_i$, so ultimately $g_1=g_2$, a contradiction.
\end{proof}

\begin{theorem}
\label{thm:self:similar:limit:sets}
The history graph $\AT$ determined by the VERS $\R(G,S)$ is the same as the self-similarity graph $\Sigma(G,S)$.
\end{theorem}

\begin{proof}
The set of vertices and of vertical edges of the two graphs are the same, so let us show that the horizontal edges are the same too.
Since there are no parallel edges with the same color by \cref{lem:self:similar:parallel:edges}, it suffices to show that $\Sigma(G,S)$ has a $g$-colored edge from $u$ to $v$ if and only if $\AT$ has that same edge.

Assume that $u \xrightarrow{g} v$ is a horizontal edge of $\AT$.
If $u$ and $v$ are words of length $1$, then they are the starting vertex $\varepsilon$ and the edge is $\varepsilon \xrightarrow{g} \varepsilon$ for some $g \in S$.
This is also an edge of $\Sigma(G,S)$, since clearly $g(\varepsilon) = \varepsilon$.
Now let us assume that the edges of $\AT$ of height at most $n$ are also edges of $\Sigma(G,S)$ and consider an edge $u \xrightarrow{g} v$ of $\AT$, where $u=iu'$ and $v=jv'$ for some $i,j \in \Alphabet$ and $u',v' \in \Lang$.
Then $u' \xrightarrow{g'} v'$ is an edge of $\AT$, for some $g' \in S$ such that $\sigma_{g'}(i)=j$ and $g_i=g'$.
By the induction hypothesis, $u' \xrightarrow{g'} v'$ is also an edge of $\Sigma(G,S)$, so $g'(u')=v'$.
All together, this implies that $g(u) = g(iu') = \sigma_g(i) g_i(u') = j g'(u') = jv' = v$, so $g(u)=v$ and thus $u \xrightarrow{g} v$ is an edge of $\Sigma(G,S)$.

Assume that $u \xrightarrow{g} v$ is a horizontal edge of $\Sigma(G,S)$.
Let $u=iu'$ and $v=jv'$ for some $i,j \in \Alphabet$ and $u',v' \in \Lang$.
Since $g(u)=v$, we have $\sigma_g(i) g_i(u') = j v'$, so $j=\sigma_g(i)$ and $v'=g_i(u')$.
Since $g_i \in S$, $\Sigma(G,S)$ has an edge $u' \xrightarrow{g_i} v'$.
By the same induction argument used in the previous paragraph, we assume that this is also an edge of $\AT$.
By the definition of $\AT$, this means that $iu' \xrightarrow{g} \sigma_g(i)v' = u \xrightarrow{g} v$ is an edge of $\AT$.
\end{proof}

As an immediate consequence, we have the following fact.

\begin{corollary}
If $G$ is a finitely generated contracting self-similar group and $S$ a finite set of generators that is closed under restrictions, the limit space of the VERS $\R(G,S)$ is the same as the limit space of $G$.
\end{corollary}

Some of these VERS are not expanding, see \cref{rmk:expanding:not:iff}.
However, even without knowing if such a VERS is expanding, we can conclude that its history graph is hyperbolic by \cref{thm:self:similar:limit:sets} together with \cite[Theorem 4.2]{LimitSets}.

These VERSs allow us to construct the Schreier graphs of any finitely generated contracting self-similar group through a simple recursive procedure, as VERS expansions.
The literature witnesses multiple ways to recursively generate such objects in many particular cases (see e.g. \cite{GLN,BGJRT}) and to generate tile graphs of bounded automata groups (see \cite{BDN}). 
However, as far as the authors know, this is the first unified procedure for such aim.

\subsection{Examples of \texorpdfstring{$\R(G,S)$}{R(G,S)}}
\label{sub:self:similar:examples}

If $\Autom$ is a finite automaton whose set $S$ of states generates a contracting self-similar group $G$, consider the following VERS.
\begin{itemize}
    \item $\V$ is a full shift, i.e., $\V=\Shift(\Sigma,t)$ for $V(\Sigma)=\{t\}$.
    \item The set of colors is $S$.
    \item The map $\kappa$ sends every $q \in S$ to $(t,t)$.
    \item For each $q \in S$, the replacement graph $R_q$ consists of the following edges:
    for each transition in $\Autom$ from $q'$ to $q$ that reads $x|y$, $R_q$ has a $q'$-colored edge from $x\mathrm{i}$ to $y\mathrm{t}$.
    \item The base graph $\Gamma_0$ consists of one $q$-colored loop for each $q \in S$.
\end{itemize}
It is readily verified that this is the VERS $\R(G,S)$ defined in \cref{sub:self:similar:VERS}.

\subsubsection{The Basilica group} \label{subsec:basilica}

The basilica group is the contracting self-similar group $\BasilicaGp$ generated by the automaton depicted in \cref{fig:basilica:group:automaton}.
Its Schreier graphs were completely classified in \cite{DDMT} and a recursive procedure to generate them can be found in \cite{BGJRT}.
The limit space of the Basilica group is homeomorphic to the well-known Basilica Julia set for the polynomial $z^2-1$, depicted in \cref{fig:basilica}.

\begin{figure}
\centering
\begin{tikzpicture}
    \node[state,gray] (id) at (0,0) {$1$};
    \node[state,red] (a) at (-2.5,1) {$a$};
    \node[state,blue] (b) at (-2.5,-1) {$b$};
    \draw[edge] (id) to[out=-30,in=30,min distance=1cm] node[right,align=center]{$0|0$\\$1|1$} (id);
    \draw[edge] (a) to[out=240,in=120] node[left]{$0|0$} (b);
    \draw[edge] (b) to[out=60,in=300] node[right]{$0|1$} (a);
    \draw[edge] (a) to[out=0,in=135] node[above]{$1|1$} (id);
    \draw[edge] (b) to[out=0,in=225] node[below]{$1|0$} (id);
\end{tikzpicture}
\caption{The automaton for the basilica group $\BasilicaGp$.}
\label{fig:basilica:group:automaton}
\end{figure}
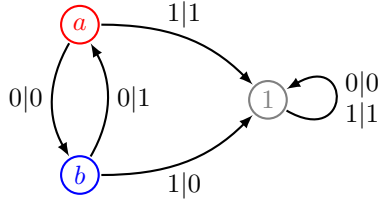

\begin{figure}
\centering
\includegraphics[width=.5\textwidth]{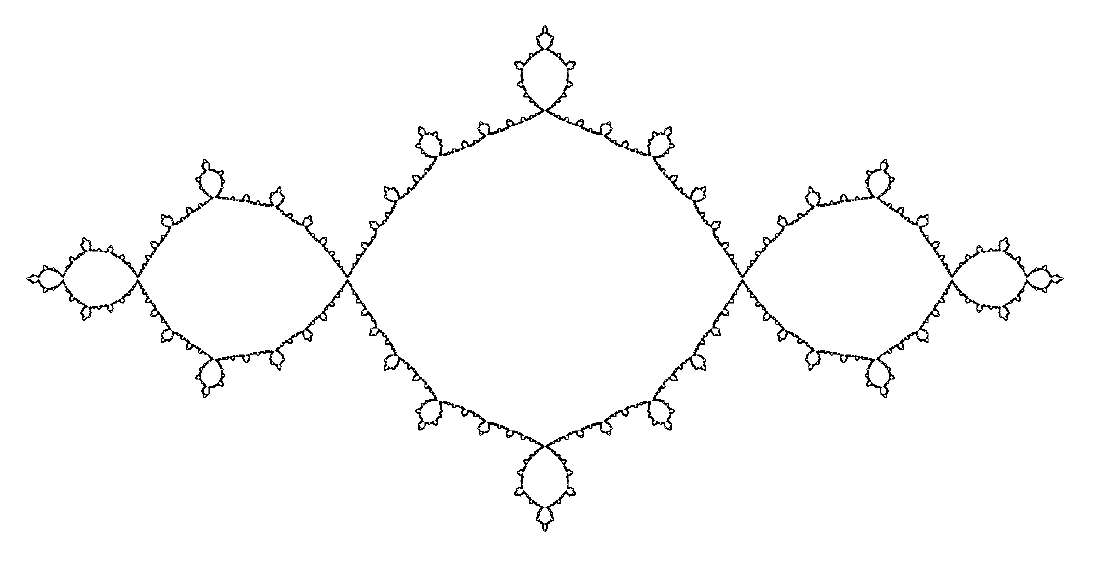}
\caption{The basilica Julia set (\href{https://commons.wikimedia.org/wiki/File:Julia_z2-1.png}{image by Prokofiev}).}
\label{fig:basilica}
\end{figure}

The replacement graphs of the associated VERS $\R(\BasilicaGp,S)$ are depicted in \cref{fig:basilica:group:VERS}.
Since every $1$-colored edge is a loop, we depict the replacement graph $R_1$ directly in the case in which the $1$-colored edge being expanded is a loop, with $\mathrm{i}=\mathrm{t}$ denoted by $\mathrm{x}$ instead.

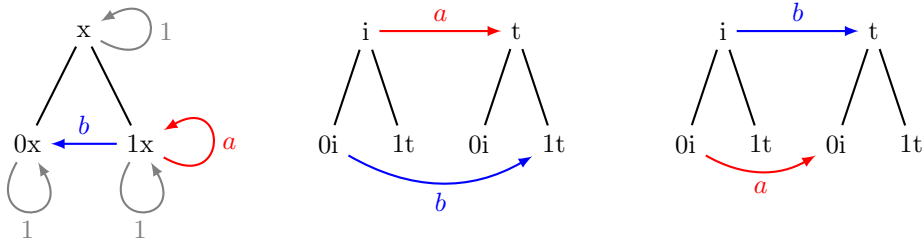
\begin{figure}
\centering
\begin{tikzpicture}
    \begin{scope}[xshift=-4.75cm]
    \node (it) at (0,0) {$\mathrm{x}$};
    \node (0it) at (-.75,-1.5) {$0\mathrm{x}$};
    \node (1it) at (.75,-1.5) {$1\mathrm{x}$};
    \draw[edge,gray] (it) to[out=-30,in=30,min distance=1cm] node[right]{$1$} (it);
    \draw (it) to (0it);
    \draw (it) to (1it);
    \draw[edge,gray] (0it) to[out=240,in=300,min distance=1cm] node[below]{$1$} (0it);
    \draw[edge,red] (1it) to[out=-30,in=30,min distance=1cm] node[right]{$a$} (1it);
    \draw[edge,gray] (1it) to[out=240,in=300,min distance=1cm] node[below]{$1$} (1it);
    \draw[edge,blue] (1it) to node[above]{$b$} (0it);
    \end{scope}
    \node (i) at (-1,0) {$\mathrm{i}$};
    \node (t) at (1,0) {$\mathrm{t}$};
    \node (0i) at (-1.5,-1.5) {$0\mathrm{i}$};
    \node (1i) at (-.5,-1.5) {$1\mathrm{t}$};
    \node (0t) at (.5,-1.5) {$0\mathrm{i}$};
    \node (1t) at (1.5,-1.5) {$1\mathrm{t}$};
    \draw[edge,red] (i) to node[above]{$a$} (t);
    \draw (i) to (0i);
    \draw (i) to (1i);
    \draw (t) to (0t);
    \draw (t) to (1t);
    \draw[edge,blue] (0i) to[out=-30,in=210] node[below]{$b$} (1t);
    \begin{scope}[xshift=4.75cm]
    \node (i) at (-1,0) {$\mathrm{i}$};
    \node (t) at (1,0) {$\mathrm{t}$};
    \node (0i) at (-1.5,-1.5) {$0\mathrm{i}$};
    \node (1i) at (-.5,-1.5) {$1\mathrm{t}$};
    \node (0t) at (.5,-1.5) {$0\mathrm{i}$};
    \node (1t) at (1.5,-1.5) {$1\mathrm{t}$};
    \draw[edge,blue] (i) to node[above]{$b$} (t);
    \draw (i) to (0i);
    \draw (i) to (1i);
    \draw (t) to (0t);
    \draw (t) to (1t);
    \draw[edge,red] (0i) to[out=-30,in=210] node[below]{$a$} (0t);
    \end{scope}
\end{tikzpicture}
\caption{The replacement graphs of the VERS $\R(\BasilicaGp,S)$.}
\label{fig:basilica:group:VERS}
\end{figure}

\subsubsection{The Grigorchuk group}

The Grigorchuk group is the contracting self-similar group $\Grig$ generated by the automaton depicted in \cref{fig:grigorchuk:group:automaton}.
Its Schreier graphs were described in \cite{Vorobets} and they were associated to a substitution shift which describes symbolically replacement rules that generate the graphs, see \cite{GLN,GLN2}.
In this case, the limit space is homeomorphic to the unit interval $[0,1]$.

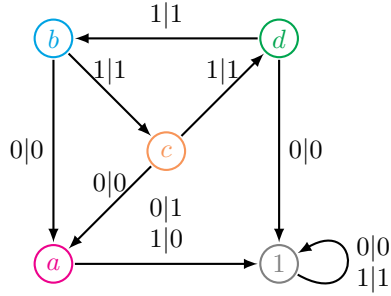
\begin{figure}
\centering
\begin{tikzpicture}
    \node[state,Magenta] (a) at (-1.5,-1.5) {$a$};
    \node[state,Cerulean] (b) at (-1.5,1.5) {$b$};
    \node[state,Peach] (c) at (0,0) {$c$};
    \node[state,Green] (d) at (1.5,1.5) {$d$};
    \node[state,gray] (id) at (1.5,-1.5) {$1$};
    \draw[edge] (id) to[out=-30,in=30,min distance=1cm] node[right,align=center]{$0|0$\\$1|1$} (id);
    \draw[edge] (a) to node[above,align=center]{$0|1$\\$1|0$} (id);
    \draw[edge] (b) to node[left]{$0|0$} (a);
    \draw[edge] (b) to node[above]{$1|1$} (c);
    \draw[edge] (c) to node[above]{$0|0$} (a);
    \draw[edge] (c) to node[above]{$1|1$} (d);
    \draw[edge] (d) to node[right]{$0|0$} (id);
    \draw[edge] (d) to node[above]{$1|1$} (b);
\end{tikzpicture}
\caption{The automaton for the Grigorchuk group $\Grig$.}
\label{fig:grigorchuk:group:automaton}
\end{figure}

The replacement graphs of the associated VERS $\R(\Grig,S)$ are depicted in \cref{fig:grigorchuk:group:VERS}.
Once again, the replacement graph $R_1$ is depicted in the case in which the $1$-colored edge being expanded is a loop.

\begin{figure}
\centering
\begin{tikzpicture}
    \begin{scope}[xshift=-3.75cm,yshift=-2cm]
    \node (i) at (-1,0) {$\mathrm{i}$};
    \node (t) at (1,0) {$\mathrm{t}$};
    \node (0i) at (-1.5,-1.5) {$0\mathrm{i}$};
    \node (1i) at (-.5,-1.5) {$1\mathrm{t}$};
    \node (0t) at (.5,-1.5) {$0\mathrm{i}$};
    \node (1t) at (1.5,-1.5) {$1\mathrm{t}$};
    \draw[edge,Magenta] (i) to node[above]{$a$} (t);
    \draw (i) to (0i);
    \draw (i) to (1i);
    \draw (t) to (0t);
    \draw (t) to (1t);
    \draw[edge,Cerulean] (0i) to[out=335,in=205] node[below]{$b$} (0t);
    \draw[edge,Peach] (0i) to[out=275,in=265] node[below]{$c$} (0t);
    \end{scope}
    \begin{scope}[xshift=-3.75cm,yshift=2cm]
    \node (i) at (-1,0) {$\mathrm{i}$};
    \node (t) at (1,0) {$\mathrm{t}$};
    \node (0i) at (-1.5,-1.5) {$0\mathrm{i}$};
    \node (1i) at (-.5,-1.5) {$1\mathrm{t}$};
    \node (0t) at (.5,-1.5) {$0\mathrm{i}$};
    \node (1t) at (1.5,-1.5) {$1\mathrm{t}$};
    \draw[edge,Cerulean] (i) to node[above]{$b$} (t);
    \draw (i) to (0i);
    \draw (i) to (1i);
    \draw (t) to (0t);
    \draw (t) to (1t);
    \draw[edge,Green] (1i) to[out=-30,in=210] node[below]{$d$} (1t);
    \end{scope}
    \begin{scope}
    \node (i) at (-1,0) {$\mathrm{i}$};
    \node (t) at (1,0) {$\mathrm{t}$};
    \node (0i) at (-1.5,-1.5) {$0\mathrm{i}$};
    \node (1i) at (-.5,-1.5) {$1\mathrm{t}$};
    \node (0t) at (.5,-1.5) {$0\mathrm{i}$};
    \node (1t) at (1.5,-1.5) {$1\mathrm{t}$};
    \draw[edge,Peach] (i) to node[above]{$c$} (t);
    \draw (i) to (0i);
    \draw (i) to (1i);
    \draw (t) to (0t);
    \draw (t) to (1t);
    \draw[edge,Cerulean] (1i) to[out=-30,in=210] node[below]{$b$} (1t);
    \end{scope}
    \begin{scope}[xshift=3.75cm,yshift=2cm]
    \node (i) at (-1,0) {$\mathrm{i}$};
    \node (t) at (1,0) {$\mathrm{t}$};
    \node (0i) at (-1.5,-1.5) {$0\mathrm{i}$};
    \node (1i) at (-.5,-1.5) {$1\mathrm{t}$};
    \node (0t) at (.5,-1.5) {$0\mathrm{i}$};
    \node (1t) at (1.5,-1.5) {$1\mathrm{t}$};
    \draw[edge,Green] (i) to node[above]{$d$} (t);
    \draw (i) to (0i);
    \draw (i) to (1i);
    \draw (t) to (0t);
    \draw (t) to (1t);
    \draw[edge,Peach] (1i) to[out=-30,in=210] node[below]{$c$} (1t);
    \end{scope}
    \begin{scope}[xshift=3.75cm,yshift=-2cm]
    \node (it) at (0,0) {$\mathrm{x}$};
    \node (0it) at (-.75,-1.5) {$0\mathrm{x}$};
    \node (1it) at (.75,-1.5) {$1\mathrm{x}$};
    \draw[edge,gray] (it) to[out=-30,in=30,min distance=1cm] node[right]{$1$} (it);
    \draw (it) to (0it);
    \draw (it) to (1it);
    \draw[edge,gray] (0it) to[out=240,in=300,min distance=1cm] node[below]{$1$} (0it);
    \draw[edge,gray] (1it) to[out=240,in=300,min distance=1cm] node[below]{$1$} (1it);
    \draw[edge,Green] (0it) to[out=150,in=210,min distance=1cm] node[left]{$d$} (0it);
    \draw[edge,Magenta] (0it) to[out=30,in=150] node[above]{$a$} (1it);
    \draw[edge,Magenta] (1it) to[out=210,in=330] node[above]{$a$} (0it);
    \end{scope}
\end{tikzpicture}
\caption{The replacement graphs of the VERS $\R(\Grig,S)$.}
\label{fig:grigorchuk:group:VERS}
\end{figure}


\section{History graphs of iterated function systems}
\label{sec:IFS}

In this section we associate to any injective post-critically finite IFS a VERS whose limit space is homeomorphic to the attractor of the IFS.

\subsection{Background on injective post-critically finite IFSs}
\label{sub:IFS:background}

Let us briefly revise the basics of IFSs and their attractors.
A standard reference is \cite[Chapter 9]{Falcon}.

Throughout this section, let $X$ be a compact subspace of $\mathbb{R}^m$ for some $m\geq1$ and let $\Alphabet = \{1,\dots,N\}$ for some $N \in \mathbb{N}_{\geq1}$.

A map $\phi \colon X \to X$ is \textbf{contracting} if there exists some $0 < C < 1$ (called \textbf{contracting ratio}) such that
$$\forall p,q \in X,\ \dist_X \left( p\phi, q\phi\right) \leq C \cdot \dist_X(p,q).$$

An \textbf{iterated function system} (\textbf{IFS} for short) is a finite set $\Phi = \{ \phi_i \}_{i\in\Alphabet}$ of contracting maps $\phi_i \colon X \to X$.
If $C_1, \dots, C_N$ are contracting ratios for $\phi_1, \dots, \phi_N$, respectively, then we say that $\max\{C_1,\dots,C_N\}$ is a \textbf{contracting ratio} of $\Phi$.

The foundational fact in the theory of IFSs is that, for each IFS $\Phi = \{ \phi_i \}_{i\in\Alphabet}$, there exists a unique non-empty compact subset $K \subseteq X$ such that $K = \bigcup_{i\in\Alphabet} (K)\phi_i$, which is called the \textbf{attractor} of $\Phi$.
It is also a standard fact that, for every non-empty compact $A \subseteq X$, one has $K = \lim_{n \to \infty} \left( (A)F_\Phi^n \right)$, where $F_\Phi$ denotes the \textit{Hutchinson operator} mapping each subset $A$ of $X$ to the closure of $\bigcup_{i\in\Alphabet}(A)\phi_i$.
In particular, every point $p \in K$ has an \textbf{address} (possibly more than one) $\alpha = \dots a_2 a_1 \in \Alphabet^{-\omega}$ such that $\bigcap_{n\geq1}(A)\phi_{a_n} \dots \phi_{a_1} = \{p\}$ for every compact $A \subseteq X$.

To avoid cumbersome notation, for all $w = x_n \dots x_1 \in \Alphabet^*$ we will write $\phi_w = \phi_{x_n \dots x_1}$ to denote the composition $\phi_{x_n} \dots \phi_{x_1}$.
Since the sets $(K)\phi_w$ will often appear, we will denote them by $K_w$.
With this notation, every point $p \in K$ has one or more addresses $\alpha = \dots a_2 a_1 \in \Alphabet^{-\omega}$ such that $\bigcap_{n\geq1} K_{a_n \dots a_1} = \{p\}$.
Moreover, note that $K_{xw} = (K_x)\phi_w \subset (K)\phi_w = K_w$ for all $x \in \Alphabet$ and $w \in \Alphabet^*$.

As done in \cite{Kigami}, let $\chi \colon \Alphabet^{-\omega} \to K$ be the map that sends each address $\dots a_2 a_1 \in \Alphabet^{-\omega}$ to the unique point in the intersection $\bigcap_{i \geq 1} K_{a_i}$.
The \textbf{critical points} of $\Phi$ are the elements of the set $\Crit \coloneqq \bigcup_{i \neq j} \left(K_i \cap K_j\right) \subseteq K$ and the \textbf{critical symbols} are the elements of $\mathcal{C} \coloneqq (\Crit)\chi^{-1}$.
We define the set of \textbf{post-critical symbols} as $\mathcal{PC} \coloneqq \bigcup_{n\geq1}\left(\mathcal{C}\right)\sigma^n$, where $\sigma$ is the shift map $(\ldots i_3 i_2 i_1)\sigma=\ldots i_3 i_2$.
The IFS $\Phi$ is said to be \textbf{post-critically finite} (\textbf{pcf} for short) if $\mathcal{PC}$ is finite.
If the IFS is injective, the set $\mathcal{PC}$ describes all possible pre-images of $\mathcal{C}$ via $\Phi$.
The \textbf{post-critical points} are the elements of the set $\PCrit \coloneqq \left(\mathcal{PC}\right)\chi$.

Note that, if $\Phi$ is pcf, then the connected components of $K$ are finitely ramified fractals in the sense of \cite{T08}, where the cells are $K_w$ for all $w \in \Alphabet^*$.

\begin{example}
\label{ex:IFS:Sierpiński:triangle}
The Sierpiński triangle (\cref{fig:Sierpiński:triangle}) can be realized as the attractor of the IFS $\ST = \{\phi_1, \phi_2, \phi_3\}$, with
$$(x,y)\phi_1 = \frac{1}{2}(x,y),\;
(x,y)\phi_2 = \frac{1}{2}(x,y)+\left(\frac{1}{2},0\right),\;
(x,y)\phi_3 = \frac{1}{2}(x,y)+\left(\frac{1}{4},\frac{\sqrt{3}}{4}\right).$$
This IFS is clearly injective.
Its critical points are $a=(\frac{1}{2},0)$, $b=(\frac{1}{4},\frac{\sqrt{3}}{4})$ and $c=(\frac{3}{4},\frac{\sqrt{3}}{4})$, and the critical symbols are $1^{-\omega}2$, $2^{-\omega}1$, $2^{-\omega}3$, $3^{-\omega}2$, $1^{-\omega}3$ and $3^{-\omega}1$.
Its post-critical symbols are $1^{-\omega}$, $2^{-\omega}$ and $3^{-\omega}$, so the post-critical points are $l = (0,0)$, $r = (1,0)$ and $t = (\frac{1}{2},\frac{\sqrt{3}}{2})$ and thus it is post-critically finite.
\end{example}

\begin{figure}
\centering
\includegraphics[width=.3333\textwidth]{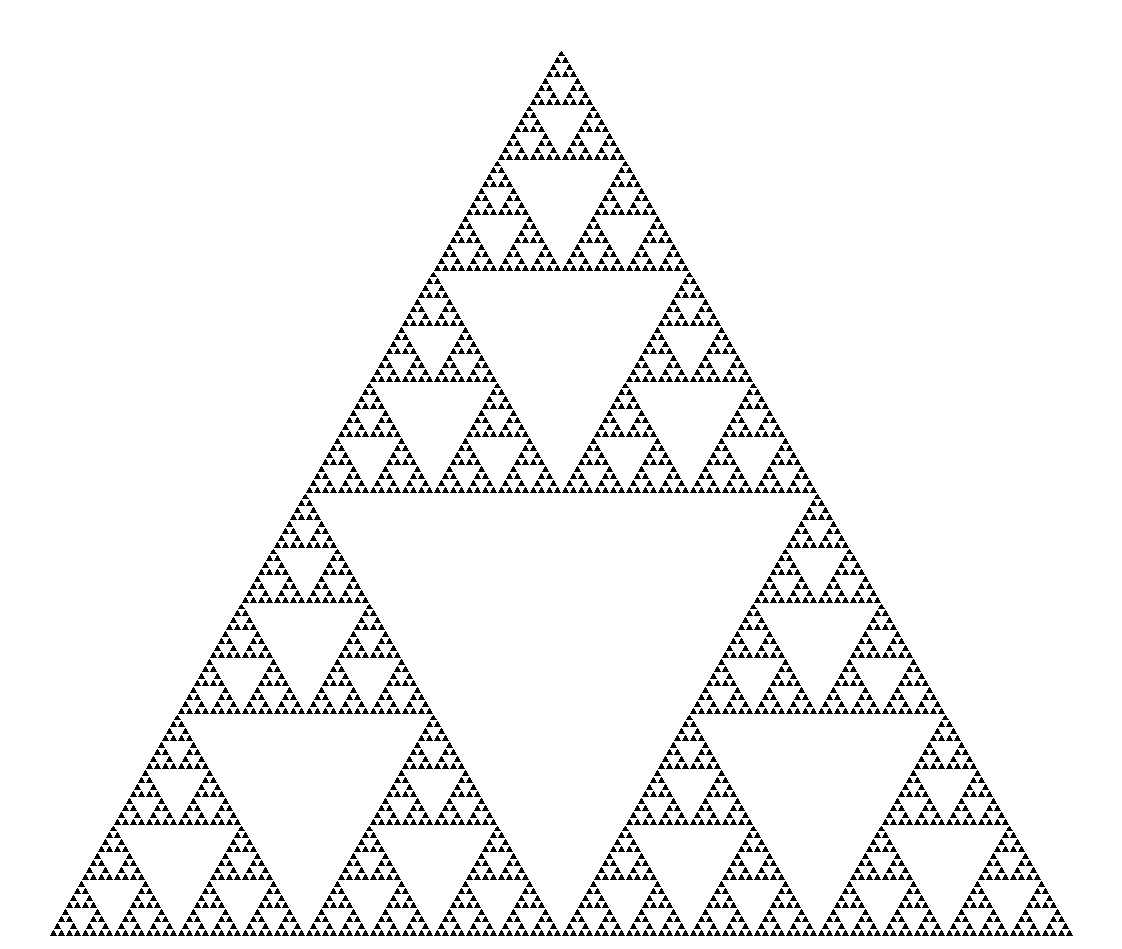}
\caption{The Sierpiński triangle.}
\label{fig:Sierpiński:triangle}
\end{figure}

\subsection{VERSs for injective pcf IFSs}

Post-critical finiteness helps us designing a VERS that keeps track of the dynamic of the IFS. Namely, each vertex represents some $K_w$ and each edge represents a point of intersection between two such sets.

To any injective pcf IFS $\Phi=\{\phi_i\}_{i=1}^N$ we associate the following VERS $\R_\Phi$.
\begin{itemize}
    \item $\V$ is the full shift on $\Alphabet$ (i.e., $V(\Sigma)=\{s\}$).
    \item The set of colors is $C = \{c_0\} \cup C_\Crit \cup C_\PCrit$, where
    $$C_\Crit = \left\{(x,y)_p \mid K_x \cap K_y = p \in \Crit,\ x<y \text{ in } \Alphabet \right\},$$
    $$C_{\PCrit} = \left\{ \PCritColor{p}{q}{x}{y} \mid p \in K_x \cap \PCrit,\ q \in K_y \cap \PCrit,\ x<y \text{ in } \Alphabet \right\}.$$
    \item The map $\kappa \colon C \to V(\Sigma) \times V(\Sigma)$ sends every $c \in C$ to $s$.
    \item The base graph consists of a $c_0$-colored loop.
    \item The replacement graphs $R_c$ have the following edges.
    \begin{itemize}
        \item The replacement graph $R_{c_0}$ has $c_0$-colored loops at every vertex and a $(x,y)_p$-colored edge $x\mathrm{i} \to y\mathrm{t}$ for all $x<y$ and all $p \in K_x \cap K_y$.
        \item For every $x<y$, the replacement graph $R_{(x,y)_p}$ has a $\PCritColor{p_1}{q_1}{z}{v}$-colored edge $z\mathrm{i} \to v\mathrm{t}$ for all $z<v$, all $p_1=(p)\phi_x^{-1} \in K_z$ and $q_1=(p)\phi_y^{-1} \in K_v$.
        \item For every $x<y$ and $p_1, q_1 \in K$, the replacement graph $R_{\PCritColor{p_1}{q_1}{x}{y}}$ has a $\PCritColor{p_2}{q_2}{z}{v}$-colored edge $z\mathrm{i} \to v\mathrm{t}$ for all $z<v$ and all $p_2=(p_1)\phi_x^{-1} \in K_z$ and $q_2=(q_1)\phi_y^{-1} \in K_v$.
    \end{itemize}
\end{itemize}

If $\Phi$ is an injective pcf IFS, its \textbf{history graph} is the history graph of the VERS $\R_\Phi$ (\cref{def:history:graph}).
We will denote it by $\AT_\Phi$, or simply $\AT$ when $\Phi$ is understood.

The VERS $\R_\ST$ associated to the IFS $\ST$ described in \cref{ex:IFS:Sierpiński:triangle} has colors
$$C = \{\textcolor{gray}{c_0}\} \cup C_\Crit \cup C_\PCrit = \{\textcolor{gray}{c_0}, \textcolor{cyan}{(1,2)_a}, \textcolor{Peach}{(1,3)_b}, \textcolor{OliveGreen}{(2,3)_c}, \textcolor{blue}{\PCritColor{l}{r}{1}{2}}, \textcolor{red}{\PCritColor{l}{t}{1}{3}}, \textcolor{Emerald}{\PCritColor{r}{t}{2}{3}}\}$$
and its replacement graphs are depicted in \cref{fig:Sierpiński:triangle:VERS}.
For instance, the replacement graph for $(1,2)_a$ is determined by the fact that $(a)\phi_1^{-1} = r \in K_2$ and $(a)\phi_2^{-1} = l \in K_1$, and the one for $\PCritColor{1}{2}{l}{r}$ by the fact that $(l)\phi_1^{-1} = l \in K_1$ and $(r)\phi_2^{-1} = r \in K_2$.
The replacement graph of $\textcolor{gray}{c_0}$ is depicted in the case in which the $\textcolor{gray}{c_0}$-colored edge is a loop, since this is always the case for the VERSs $\R_\Phi$ (for any IFS $\Phi$).

\begin{figure}
\centering
\begin{tikzpicture}[font=\small]
    \begin{scope}[xshift=-6.25cm,yshift=2.75cm]
    \node (it) at (0,0) {$\mathrm{x}$};
    \node (1it) at (-1.5,-1.5) {$1\mathrm{x}$};
    \node (2it) at (0,-1.5) {$2\mathrm{x}$};
    \node (3it) at (1.5,-1.5) {$3\mathrm{x}$};
    \draw[edge,gray] (it) to[out=-30,in=30,min distance=1cm] node[right]{$c_0$} (it);
    \draw (it) to (1it);
    \draw (it) to (2it);
    \draw (it) to (3it);
    \draw[edge,gray] (1it) to[out=210,in=150,min distance=1cm] node[left]{$c_0$} (1it);
    \draw[edge,gray] (2it) to[out=240,in=300,min distance=1cm] node[below]{$c_0$} (2it);
    \draw[edge,gray] (3it) to[out=-30,in=30,min distance=1cm] node[right]{$c_0$} (3it);
    \draw[edge,Cyan] (1it) to node[below]{$(1,2)_a$} (2it);
    \draw[edge,Peach] (1it) to[out=270,in=270,looseness=4/3] node[below]{$(1,3)_b$} (3it);
    \draw[edge,OliveGreen] (2it) to node[below]{$(2,3)_c$} (3it);
    \end{scope}
    \begin{scope}
    \node (i) at (-1.5,0) {$\mathrm{i}$};
    \node (1i) at (-2.5,-1.5) {$1\mathrm{i}$};
    \node (2i) at (-1.5,-1.5) {$2\mathrm{i}$};
    \node (3i) at (-.5,-1.5) {$3\mathrm{i}$};
    \node (t) at (1.5,0) {$\mathrm{t}$};
    \node (1t) at (.5,-1.5) {$1\mathrm{t}$};
    \node (2t) at (1.5,-1.5) {$2\mathrm{t}$};
    \node (3t) at (2.5,-1.5) {$3\mathrm{t}$};
    \draw[edge,Cyan] (i) to node[above]{$(1,2)_a$} (t);
    \draw (i) to (1i);
    \draw (i) to (2i);
    \draw (i) to (3i);
    \draw (t) to (1t);
    \draw (t) to (2t);
    \draw (t) to (3t);
    \draw[edge,blue] (1t) to[out=210,in=330] node[below]{$\PCritColor{1}{2}{l}{r}$} (2i);
    \end{scope}
    \begin{scope}[xshift=-6.25cm,yshift=-1.75cm]
    \node (i) at (-1.5,0) {$\mathrm{i}$};
    \node (1i) at (-2.5,-1.5) {$1\mathrm{i}$};
    \node (2i) at (-1.5,-1.5) {$2\mathrm{i}$};
    \node (3i) at (-.5,-1.5) {$3\mathrm{i}$};
    \node (t) at (1.5,0) {$\mathrm{t}$};
    \node (1t) at (.5,-1.5) {$1\mathrm{t}$};
    \node (2t) at (1.5,-1.5) {$2\mathrm{t}$};
    \node (3t) at (2.5,-1.5) {$3\mathrm{t}$};
    \draw[edge,Peach] (i) to node[above]{$(1,3)_b$} (t);
    \draw (i) to (1i);
    \draw (i) to (2i);
    \draw (i) to (3i);
    \draw (t) to (1t);
    \draw (t) to (2t);
    \draw (t) to (3t);
    \draw[edge,red] (1t) to node[below]{$\PCritColor{1}{3}{l}{t}$} (3i);
    \end{scope}
    \begin{scope}[xshift=0cm,yshift=-3.5cm]
    \node (i) at (-1.5,0) {$\mathrm{i}$};
    \node (1i) at (-2.5,-1.5) {$1\mathrm{i}$};
    \node (2i) at (-1.5,-1.5) {$2\mathrm{i}$};
    \node (3i) at (-.5,-1.5) {$3\mathrm{i}$};
    \node (t) at (1.5,0) {$\mathrm{t}$};
    \node (1t) at (.5,-1.5) {$1\mathrm{t}$};
    \node (2t) at (1.5,-1.5) {$2\mathrm{t}$};
    \node (3t) at (2.5,-1.5) {$3\mathrm{t}$};
    \draw[edge,OliveGreen] (i) to node[above]{$(2,3)_c$} (t);
    \draw (i) to (1i);
    \draw (i) to (2i);
    \draw (i) to (3i);
    \draw (t) to (1t);
    \draw (t) to (2t);
    \draw (t) to (3t);
    \draw[edge,Emerald] (2t) to[out=210,in=330] node[below]{$\PCritColor{2}{3}{r}{t}$} (3i);
    \end{scope}
    \begin{scope}[xshift=-6.25cm,yshift=-5.25cm]
    \node (i) at (-1.5,0) {$\mathrm{i}$};
    \node (1i) at (-2.5,-1.5) {$1\mathrm{i}$};
    \node (2i) at (-1.5,-1.5) {$2\mathrm{i}$};
    \node (3i) at (-.5,-1.5) {$3\mathrm{i}$};
    \node (t) at (1.5,0) {$\mathrm{t}$};
    \node (1t) at (.5,-1.5) {$1\mathrm{t}$};
    \node (2t) at (1.5,-1.5) {$2\mathrm{t}$};
    \node (3t) at (2.5,-1.5) {$3\mathrm{t}$};
    \draw[edge,blue] (i) to node[above]{$\PCritColor{1}{2}{l}{r}$} (t);
    \draw (i) to (1i);
    \draw (i) to (2i);
    \draw (i) to (3i);
    \draw (t) to (1t);
    \draw (t) to (2t);
    \draw (t) to (3t);
    \draw[edge,blue] (1i) to[out=330,in=210,looseness=2/3] node[below]{$\PCritColor{1}{2}{l}{r}$} (2t);
    \end{scope}
    \begin{scope}[xshift=0cm,yshift=-7cm]
    \node (i) at (-1.5,0) {$\mathrm{i}$};
    \node (1i) at (-2.5,-1.5) {$1\mathrm{i}$};
    \node (2i) at (-1.5,-1.5) {$2\mathrm{i}$};
    \node (3i) at (-.5,-1.5) {$3\mathrm{i}$};
    \node (t) at (1.5,0) {$\mathrm{t}$};
    \node (1t) at (.5,-1.5) {$1\mathrm{t}$};
    \node (2t) at (1.5,-1.5) {$2\mathrm{t}$};
    \node (3t) at (2.5,-1.5) {$3\mathrm{t}$};
    \draw[edge,red] (i) to node[above]{$\PCritColor{1}{3}{l}{t}$} (t);
    \draw (i) to (1i);
    \draw (i) to (2i);
    \draw (i) to (3i);
    \draw (t) to (1t);
    \draw (t) to (2t);
    \draw (t) to (3t);
    \draw[edge,red] (1i) to[out=330,in=210,looseness=2/3] node[below]{$\PCritColor{1}{3}{l}{t}$} (3t);
    \end{scope}
    \begin{scope}[xshift=-6.25cm,yshift=-8.75cm]
    \node (i) at (-1.5,0) {$\mathrm{i}$};
    \node (1i) at (-2.5,-1.5) {$1\mathrm{i}$};
    \node (2i) at (-1.5,-1.5) {$2\mathrm{i}$};
    \node (3i) at (-.5,-1.5) {$3\mathrm{i}$};
    \node (t) at (1.5,0) {$\mathrm{t}$};
    \node (1t) at (.5,-1.5) {$1\mathrm{t}$};
    \node (2t) at (1.5,-1.5) {$2\mathrm{t}$};
    \node (3t) at (2.5,-1.5) {$3\mathrm{t}$};
    \draw[edge,Emerald] (i) to node[above]{$\PCritColor{2}{3}{r}{t}$} (t);
    \draw (i) to (1i);
    \draw (i) to (2i);
    \draw (i) to (3i);
    \draw (t) to (1t);
    \draw (t) to (2t);
    \draw (t) to (3t);
    \draw[edge,Emerald] (2i) to[out=330,in=210,looseness=2/3] node[below]{$\PCritColor{2}{3}{r}{t}$} (3t);
    \end{scope}
\end{tikzpicture}
\caption{The replacement graphs of the VERS $\R_\ST$.}
\label{fig:Sierpiński:triangle:VERS}
\end{figure}
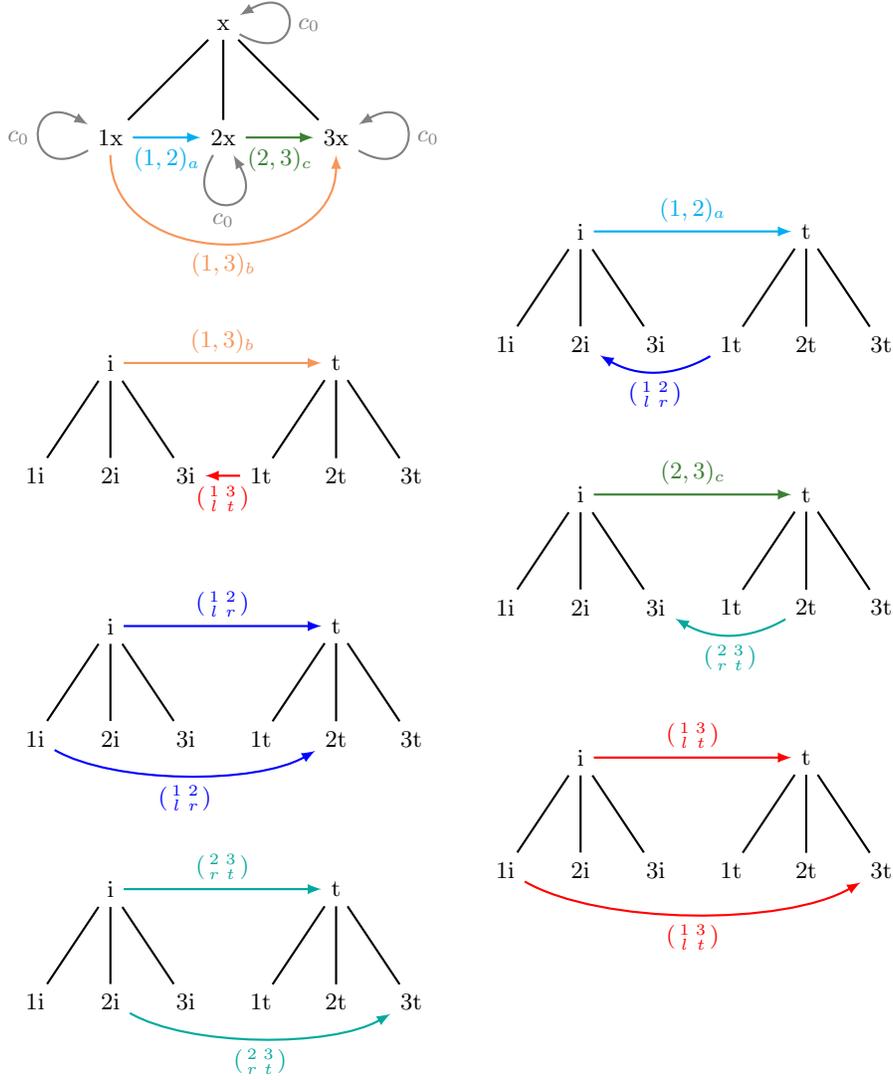

\begin{remark}
In the construction above, $\V$ is a full shift (the set of all infinite words over a finite alphabet).
It is likely that, replacing $\V$ with any edge shift, one could generalize the construction for \textit{graph directed constructions} from \cite{MauldinWilliams}.
\end{remark}

Note that this approach is not dissimilar from the recursive procedure developed in \cite{BDN} for bounded automata groups.
In particular, parallels can be drawn both between the tiles of a bounded automata group and the sets $K_w$ and between the post-critical symbols of a pcf self-similar IFS (see \cref{sub:IFS:background}) and the post-critical sequences of a bounded automata group.
For example, in the case of the Sierpiński triangle, the post-critical sequences for the Hanoi tower group (see \cite[Example 3.9.2]{NekSSG}) are the same as the post-critical symbols in \cref{ex:IFS:Sierpiński:triangle}, even if the corresponding VERSs are slightly different.

\begin{remark}
\label{rmk:IFS:colors}
Assume that $\Gamma_n$ has a $c$-colored edge $x_n \dots x_1 \to y_n \dots y_1$.
By the definition of the replacement graphs of $\R_\Phi$, an easy inductive argument shows that:
\begin{itemize}
    \item $c=c_0$ if and only if $x_i=y_i$ for all $i=1,\dots,n$;
    \item $c \in C_\Crit$ if and only if $x_i=y_i$ for all $i=1,\dots,n-1$, $x_n<y_n$ and $c=(x_n,y_n)_p$ for some $p$;
    \item $c \in C_\PCrit$ if and only $x_{n-1} \neq y_{n-1}$, $x_n < y_n$ and $c = \PCritColor{p}{q}{x_n}{y_n}$ for some $p, q$.
\end{itemize}
\end{remark}

\begin{proposition}
\label{prop:IFS:Crit:colors}
Let $n\in\mathbb{N}$, $w\in\Alphabet^n$ and $x<y$ in $\Alphabet$.
Then there exists $r \in K_{xw} \cap K_{yw}$ if and only if $\Gamma_{n+1}$ has an edge $xw \to yw$ colored by $(x,y)_p \in C_\Crit$ for some $p$, in which case $p=(r)\phi_w^{-1}$.
\end{proposition}

\begin{proof}
Let us first assume that $K_{xw} \cap K_{yw}$ includes a point $r$ and let $p=(r)\phi_w^{-1}$.
Then $p \in K_x \cap K_y$.
By the very definition of $\R_\Phi$, the graph $\Gamma_n$ (possibly $n=0$) features a $c_0$-colored loop at the vertex $w$.
Since $p \in K_x \cap K_y$, its expansion produces an $(x,y)_p$-colored edge $xw \to yw$, as needed.

Assume now that $\Gamma_{n+1}$ has a $(x,y)_p$-colored edge $xw \to yw$ and let $r=(p)\phi_w$.
By the definition of $\R_\Phi$, its spanning lift is necessarily a $c_0$-colored loop $w \to w$, so $p \in K_x \cap K_y$ and thus $r \in K_{xw} \cap K_{yw}$.
\end{proof}

\begin{proposition}
\label{prop:IFS:PCrit:colors}
Let $m,k\in\mathbb{N}_{>0}$ and consider $u = x_m \dots x_1 x_0 w$, $v = y_m \dots y_1 y_0 w$ elements of $\Alphabet^*$ for some $w \in \Alphabet^k$ and $x_i, y_i \in \Alphabet$ ($i=0\dots,m$) with $x_m < y_m$ and $x_0 \neq y_0$.
Then there exists $r \in K_u \cap K_v$ if and only if $\Gamma_{k+m+1}$ has an edge $u \to v$ that is colored by $\PCritColor{p}{q}{x_m}{y_m} \in C_\PCrit$ for some $p$ and $q$, in which case $p=(r)\phi_{x_{m-1} \dots x_1 x_0 w}^{-1}$ and $q=(r)\phi_{y_{m-1} \dots y_1 y_0 w}^{-1}$.
\end{proposition}

\begin{proof}
Let us fix $k>0$ and proceed by induction on $m>0$.
We first show both implications when $m=1$, for which $u =  x_1 x_0 w$ and $v = y_1 y_0 w$ with $x_1 < y_1$.

Assume that there exists $r \in K_u \cap K_v$.
Then $r \in K_{x_0w} \cap K_{y_0w}$.
Now, by \cref{prop:IFS:Crit:colors}, $\Gamma_{m+1}$ has an edge $x_0w \leftrightarrow y_0w$ colored by $(x_0,y_0)_{(r)\phi_w^{-1}}$ or $(y_0,x_0)_{(r)\phi_w^{-1}}$ (depending on whether $x_0<y_0$ or $y_0<x_0$).
By the definition of $\R_\Phi$, the expansion of such edge produces an edge $x_1 x_0 w \to y_1 y_0 w$ of $\Gamma_{n+2}$ that is colored by $\PCritColor{p}{q}{x_1}{y_1}$ for $p=(r)\phi_{w}^{-1}\phi_{x_0}^{-1}=(r)\phi_{x_0w}^{-1}$ and $q=(r)\phi_{w}^{-1}\phi_{y_0}^{-1}=(r)\phi_{y_0w}^{-1}$.

Conversely, assume that $\Gamma_{n+2}$ has a $\PCritColor{p}{q}{x_1}{y_1}$-colored edge $x_1 x_0 m \to y_1 y_0 w$, for some $p$ and $q$.
By \cref{rmk:IFS:colors}, its spanning lift is an edge $x_0 w \leftrightarrow y_0 w$ colored by $(x_0,y_0)_s$ for some $s$.
By definition of $\R_\Phi$ we have that $p=(s)\phi_{x_0}^{-1}\in K_{x_1}$ and $q=(s)\phi_{y_0}^{-1}\in K_{y_1}$, so $s=(r)\phi_w^{-1} \in K_{x_1 x_0} \cap K_{y_1 y_0}$.
Let $r = (s)\phi_w$.
Then $r \in K_{x_1 x_0 w} \cap K_{y_1 y_0 w}$ and $p = (r)\phi_{x_0w}^{-1}$ and $q = (r)\phi_{y_0w}^{-1}$, as needed.

Suppose now that $m\geq2$ and assume the statement for the prefixes $\hat{u} = x_{m-1} \dots x_0 w$ and $\hat{v} = y_{m-1} \dots y_0 w$.
Let us show both implications of the statement for $u=x_m\hat{u}$ and $v=y_m\hat{v}$, one at a time.

Assume that there exists $r \in K_u \cap K_v$.
Then $r \in K_{\hat{u}} \cap K_{\hat{v}}$.
By the inductive hypothesis, $\Gamma_{k+m}$ has an edge $\hat{u} \leftrightarrow \hat{v}$ colored by $\PCritColor{\hat{p}}{\hat{q}}{x_{m-1}}{y_{m-1}}$ or $\PCritColor{\hat{q}}{\hat{p}}{y_{m-1}}{x_{m-1}}$ (depending on whether $x_{m-1} < y_{m-1}$ or $y_{m-1}<x_{m-1}$), where $\hat{p} = (r)\phi^{-1}_{x_{m-2} \dots x_0 w}$ and $\hat{q} = (r)\phi^{-1}_{y_{m-2} \dots y_0 w}$.
Let $p=(r)\phi_{x_{m-1} \dots x_1 x_0 w}^{-1}$ and $q=(r)\phi_{y_{m-1} \dots y_1 y_0 w}^{-1}$ and note that $p = (\hat{p})\phi_{x_{m-1}}^{-1} \in K_{x_m}$ and $q = (\hat{q})\phi_{y_{m-1}}^{-1} \in K_{y_m}$.
Then, by the definition of $\R_\Phi$, the expansion of the aforementioned edge $\hat{u} \leftrightarrow \hat{v}$ produces an edge $x_m \hat{u} \to y_m \hat{v}$ (i.e., $u \to v$) colored by $\PCritColor{p}{q}{x_m}{y_m}$, since $x_m < y_m$.

Conversely, assume that $\Gamma_{k+m+1}$ has a $\PCritColor{p}{q}{x_m}{y_m}$-colored edge $u \to v$ for some $p$ and $q$.
By \cref{rmk:IFS:colors}, since $m\geq2$ and $x_{m-2} \neq y_{m-2}$, its spanning lift is an edge $x_{m-1} \dots x_0 w \leftrightarrow y_{m-1} \dots y_0 w$ colored by $\PCritColor{\hat{p}}{\hat{q}}{x_{m-1}}{y_{m-1}}$ for some $\hat{p}$ and $\hat{q}$.
By the induction hypothesis, there exists $r \in K_{\hat{u}} \cap K_{\hat{v}}$, which is such that $\hat{p} = (r)\phi_{x_{m-2} \dots x_0 w}^{-1}$ and $\hat{q} = (r)\phi_{y_{m-2} \dots y_0 w}^{-1}$.
By definition of $\R_\Phi$ we have that $p = (\hat{p})\phi_{x_{m-1}}^{-1} \in K_{x_m}$ and that $q = (\hat{q})\phi_{y_{m-1}}^{-1} \in K_{y_m}$.
Putting the last two sentences together, we obtain that
$$p = \left((r)\phi_{x_{m-2} \dots x_0 w}^{-1}\right)\phi_{x_{m-1}}^{-1} = (r)\phi_{x_{m-1} \dots x_0 w} \in K_{x_m},$$
$$q = \left((r)\phi_{y_{m-2} \dots y_0 w}^{-1}\right)\phi_{y_{m-1}}^{-1} = (r)\phi_{y_{m-1} \dots y_0 w} \in K_{y_m},$$
so $r \in K_u \cap K_v$, all as needed.
\end{proof}

\begin{remark}
\label{rmk:IFS:Crit:PCrit:colors}
By \cref{prop:IFS:Crit:colors,,prop:IFS:PCrit:colors}, each edge $u \leftrightarrow v$ whose color lies in $C_\Crit \cup C_\PCrit$ identifies a point that lies in $K_u \cap K_v$, denoted by $r$ in both propositions.
In the proof of \cref{prop:IFS:PCrit:colors} we actually showed that $r$ remains the same when expanding.
More precisely, let $e_1$ and $e_2$ be edges whose colors lie in $C_\Crit \cup C_\PCrit$ and let $r_1$ and $r_2$ be the points associated (in the sense described just above) to the edges $e_1$ and $e_2$, respectively;
if $e_1$ is the spanning lift of $e_2$, then $r_1 = r_2$.
\end{remark}

\cref{prop:IFS:Crit:colors,,prop:IFS:PCrit:colors} immediately show the following crucial fact.

\begin{theorem}
\label{thm:IFS:history:graph}
For all $n \in \mathbb{N}$, the graph $\Gamma_n$ has an edge $u \leftrightarrow v$ if and only if $K_u \cap K_v$ is non-empty.
\end{theorem}

Note that, as a consequence, for a self-similar pcf IFS whose maps have the same similarity ratio, the history graph $\AT_\Phi$ is the same as the augmented tree considered in \cite{KLLW}.

\subsection{Powers of IFSs and their history graphs}
\label{sub:power}

Let us introduce a few tools that will be useful to show that the history graph is hyperbolic in the next subsection.

Let $\Phi$ be an IFS and let $k \in \mathbb{N}_{\geq1}$.
The \textbf{$k$-th power} of $\Phi$ is the IFS $\Phi^k = \left\{ \phi_w \mid w \in \Alphabet^k \right\}$.
This is just a recoding of the IFS into an alphabet of words of the given length $k$.
It is clear that the attractors of $\Phi$ and $\Phi^k$ are the same.

A notion of power can also be defined for augmented trees as follows.

\begin{definition}
\label{def:augmented:tree:power}
Let $\AT$ be an augmented tree with root $v_0$ and let $k \in \mathbb{N}_{\geq0}$.
The \textbf{$k$-th power} of $\AT$ is the graph $\AT^k$ defined as follows:
\begin{itemize}
    \item $V(\AT^k) = \{ v \in \AT \mid \dist_{\AT}(v_0,v) \equiv 0\ \mathrm{mod}\ k \}$;
    \item the vertical edges of $\AT^k$ are $u \to v$ whenever $\dist_{\AT}(u,v_0) < \dist_{\AT}(v,v_0)$ and  $\dist_{\AT}(u,v)=k$;
    \item the horizontal edges $\AT^k$ are the same as those of $\AT$, whenever they join vertices that belong to $\AT^k$ too.
\end{itemize}
\end{definition}

\begin{remark}
\label{rmk:IFS:power:history:graph}
Let $\Phi$ be an injective pcf IFS and let $k \in \mathbb{N}_{\geq0}$.
Then it is clear that the history graph of $\Phi^k$ is the same as the power of the history graph of $\Phi$.
\end{remark}

The following fact is a generalization of point 2 of Lemma 4.1 of \cite{NekSSG}.

\begin{lemma}
\label{lem:IFS:power:quasi:isometry}
An augmented tree $\mathcal{AT}$ is quasi-isometric to $\mathcal{AT}^k$ for any $k \in \mathbb{N}$.
\end{lemma}

\begin{proof}
Consider the map that associates to each vertex of $\mathcal{AT}^{k}$ the corresponding vertex in $\mathcal{AT}$.
Quasi-density follows by considering, for any vertex of $\mathcal{AT}$, the first of its predecessors belonging to $\mathcal{AT}^k$:
their distance is smaller than $k$.

Now let us prove that $\dist_{\mathcal{AT}}(p,q) \leq k\cdot\dist_{\mathcal{AT}^k}(p,q)$.
Consider a geodesic joining $p$ and $q$ in $\mathcal{AT}^k$:
it corresponds to a path $\gamma$ in $\mathcal{AT}$
where the vertical edges come in multiples of $k$, each multiple contributing $1$ to the distance in $\AT^k$.
From this observation, we obtain the desired inequality as follows:
$$\dist_{\mathcal{AT}}(p,q) \leq \text{ length of } \gamma \leq k\cdot\dist_{\mathcal{AT}^k}(p,q).$$

For the other inequality, let $\gamma$ be a geodesic in $\AT$ and let us modify it as follows.
We keep the vertical edges unchanged and we replace each horizontal edge $e$ with a path $e'$ comprised of three pieces described right below.
Say that $e$ is $u \leftrightarrow v$ and let $u'$ and $v'$ be the respective closest predecessors that belong to $\AT^k$.
Then the path $e'$ has a vertical part joining $u$ to $u'$, a horizontal edge $u' \leftrightarrow v'$ (which exists because $\AT$ is an augmented tree) and a vertical part joining $v'$ with $v$.
Let $\gamma'$ denote the result of this modification.
Note that $\gamma'$ corresponds to a path in $\AT^k$, so $\dist_{\mathcal{AT}^k}(p,q) \leq \text{length}(\gamma')$.
Moreover, by construction the length of $\gamma'$ is at most $(2k+1) \cdot\dist_{\mathcal{AT}}(p,q)$.
Hence we ultimately have that $\dist_{\mathcal{AT}^k}(p,q) \leq (2k+1) \cdot \dist_{\mathcal{AT}}(p,q)$, which concludes the proof.
\end{proof}

\begin{remark}
\label{rmk:IFS:power:contracting:ratios}
The composition of two contractions $f_1$ and $f_2$ is a contraction whose ratio is at most the product of the ratios of $f_1$ anb $f_2$.
Thus, if $k \geq 2$ then the contracting ratio of $\Phi^k$ is strictly smaller than the one of $\Phi$.

\end{remark}

\begin{lemma}
\label{lem:IFS:power:postcrit}
For every IFS $\Phi$ and all $k\in\mathbb{N}_{\geq0}$, we have that $\PCrit_{\Phi^k} \subseteq \PCrit_\Phi$.
\end{lemma}

\begin{proof}
We will make use of the following more general claim.

\begin{claim}
Let $u,v,w \in \Alphabet^*$ with $|u|=|v|\leq|w|$.
Let $p \in K$ and assume that $q = (p)\phi_w \in K_u \cap K_v$.
Then $(q)\phi_w^{-1} \in \PCrit_\Phi$.
\end{claim}

Let us first see that this claim implies the Lemma.
If $p \in \PCrit_{\Phi^k}$, then $p=(q)\phi_w^{-1}$ for some $w \in {}^{*}_{}{\left(\Alphabet^k\right)}$ and $q \in \Crit_{\Phi_k}$, i.e., $q \in K_u \cap K_v$ for some $u,v \in \Alphabet^k$.
In particular, $|u|=|v|\leq|w|$ and $p$ and $q$ are as in the claim, so we can conclude that $p \in \PCrit_\Phi$, as needed.

We now prove the claim by induction on $m=|u|=|v|$.
If $m=1$ then $u,v\in\Alphabet$, so $q \in K_u \cap K_v$ implies that $q \in \Crit_\Phi$ and thus $p = (q)\phi_w^{-1} \in \PCrit_\Phi$, as needed.
Let us now assume that $m \geq 2$.
If $q \in \Crit_\Phi$, then once again $p = (q)\phi_w^{-1} \in \PCrit_\Phi$, so let us assume that $q \not\in \Crit_\Phi$.
Then there must exist a unique $x \in \Alphabet$ such that $q \in K_x$.
Hence, since $q \in K_u \cap K_v \cap K_2$, we have that $u=\hat{u}x$, $v=\hat{v}x$ and $w=\hat{w}x$, where $|\hat{u}|=|\hat{v}|=m-1$ and $|\hat{w}| = |w|-1 \geq m-1 \geq 1$, so $w \in \Alphabet^*$.
Let $\hat{q}=(q)\phi_x^{-1} \in K_{\hat{w}}$.
Then $\hat{q} \in K_{\hat{u}} \cap K_{\hat{v}}$ and $\hat{q} = (q)\phi_x^{-1} = \left((p)\phi_w\right)\phi_x^{-1} = (p)\phi_{\hat{w}}\phi_x\phi_x^{-1} = (p)\phi_{\hat{w}}$.
We can then apply the inductive hypothesis on $\hat{q}$, concluding that $p = (\hat{q})\phi_{\hat{w}}^{-1} \in \PCrit_\Phi$.
\end{proof}

\subsection{Hyperbolicity of the history graph}

Let us first see that a simple condition on the contracting ratio is sufficient for the history graph to be hyperbolic.

\begin{proposition}
\label{prop:IFS:no:2:squares}
If the contracting ratio of an injective pcf IFS is smaller than
$$\min\left\{ \dist(p,q) \mid p,q\in\PCrit \right\} / \max\left\{ \dist(p,q) \mid p,q\in\PCrit \right\},$$
then its history graph has no geodesic $2$-squares.

\end{proposition}

\begin{proof}
Assume by contradiction that $\AT$ has a geodesic $2$-square.
Reasoning as in the proof of \cref{lem:squares:lift}, we can take the spanning lifts of the bottom horizontal edges and then their spanning lifts, obtaining a ``grid'' with vertices $u$, $v$ and $w$ at the top, $xu$, $yv$ and $zw$ in the middle and $\hat{x}xu$, $\hat{y}yv$ and $\hat{z}zw$ at the bottom, where the outer edges of the grid form a geodesic $2$-square.
By \cref{rmk:IFS:colors}, the horizontal edges of the square have colors in $C_\Crit \cup C_\PCrit$.
The intermediate horizontal edges $xu \leftrightarrow yv$ and $yv \leftrightarrow zw$ are the spanning lifts of the bottom horizontal edges $\hat{x}xu \leftrightarrow \hat{y}yv$, respectively, and the top horizontal edges $u \leftrightarrow v$ and $v \leftrightarrow w$ are the spanning lifts of $xu \leftrightarrow yv$ and $yv \leftrightarrow zw$, respectively.
Then, by definition of $\R_\Phi$, the four horizontal edges of the bottom subrectangle have colors in $C_\PCrit$, so ultimately the grid must have the following shape:
\begin{center}
\begin{tikzcd}
    u \arrow[r,leftrightarrow] \arrow[d,-stealth] & v \arrow[r,leftrightarrow] \arrow[d,-stealth] & w \arrow[d,-stealth]&\\
    xu \arrow[r,leftrightarrow,"\PCritColor{a}{b}{x}{y}"] \arrow[d,-stealth] & yv \arrow[r,leftrightarrow,"\PCritColor{c}{d}{y}{z}"] \arrow[d,-stealth] & zw \arrow[d,-stealth]&\\
    \hat{x}xu \arrow[r,leftrightarrow,"\PCritColor{\hat{a}}{\hat{b}}{\hat{x}}{\hat{y}}"] & \hat{y}yv \arrow[r,leftrightarrow,"\PCritColor{\hat{c}}{\hat{d}}{\hat{y}}{\hat{z}}"] & \hat{z}zw &
\end{tikzcd}
\end{center}

Since the horizontal edges of the bottom subrectangle have colors in $C_\PCrit$, we can apply \cref{prop:IFS:PCrit:colors,,rmk:IFS:Crit:PCrit:colors}, obtaining that there exist $r \in K_{\hat{x}xu} \cap K_{\hat{y}yv}$ and $s \in K_{\hat{y}yv} \cap K_{\hat{z}zw}$ such that
$$r = (a)\phi_u = (b)\phi_v = (\hat{a})\phi_{xu} = (\hat{b})\phi_{yv},\ s = (\hat{c})\phi_v = (\hat{d})\phi_w = (\hat{c})\phi_{yv} = (\hat{d})\phi_{zw}.$$
In particular, $b = (\hat{b})\phi_y$ and $c = (\hat{c})\phi_y$.
For all $i \in \Alphabet$, let $C_i$ be the contracting ratio of $\phi_i$ and let $C = \max \left\{ C_i \mid i \in \Alphabet \right\}$.
Using the hypothesis, we have that
\begin{multline*}
\dist(b,c) = \dist\left( (\hat{b})\phi_y, (\hat{c})\phi_y \right) \leq C_y \cdot \dist(\hat{b},\hat{c}) \leq C \cdot \dist(\hat{b},\hat{c}) <\\
<\frac{\min\left\{ \dist(p,q) \mid p,q\in\PCrit \right\}}{\max\left\{ \dist(p,q) \mid p,q\in\PCrit \right\}} \cdot \dist(\hat{b},\hat{c}) \leq \frac{\dist(b,c)}{\dist(\hat{b},\hat{c})} \cdot \dist(\hat{b},\hat{c}) = \dist(b,c),
\end{multline*}
which is a contradiction.
Thus, $\AT_\Phi$ cannot have geodesic $2$-squares, as needed.
\end{proof}

Let us now put together the results of this and the previous subsections.

\begin{theorem}
\label{thm:IFS:hyperbolic}
The history graph of an injective pcf IFS is hyperbolic.
\end{theorem}

\begin{proof}
Recall that having no geodesic $2$-squares is sufficient to show that $\AT_\Phi$ is hyperbolic by \cref{thm:Kaimano}.
Then, by \cref{lem:IFS:power:quasi:isometry,,prop:IFS:no:2:squares}, it suffices to show that some power of $\Phi$ is such that the maximum of all its contracting ratios is smaller than the ratio
$$R \coloneqq \min\left\{ \dist(p,q) \mid p,q\in\PCrit \right\} / \max\left\{ \dist(p,q) \mid p,q\in\PCrit \right\}.$$
By \cref{lem:IFS:power:postcrit}, taking powers of an IFS either preserves or removes post-critical points, so $R$ does not decrease when taking powers.
Powers strictly decrease the maximum of the contracting ratios (\cref{rmk:IFS:power:contracting:ratios}), so there exists some $k$ such that $\Phi^k$ is as desired and we are done.
\end{proof}

\subsection{Identification of the attractor with the VERS limit space}

Consider the map $\chi$ from the set of geodesic rays of $\AT_\Phi$ to $K$ mapping each $\alpha$ with vertices $\dots a_2 a_1 \in \Alphabet^{-\omega}$ to the unique point in the intersection $\bigcap_{i \geq 1} K_{a_i}$.

\begin{lemma}
\label{lem:IFS:quotient:map:well:defined}
If $\alpha$ and $\beta$ are asymptotically equivalent geodesic rays of $\AT_\Phi$, then $(\alpha)\chi = (\beta)\chi$.
\end{lemma}

Even if our augmented tree $\AT_\Phi$ is different from the one in \cite{SSIFS}, the arguments of its Lemma 4.1 hold almost verbatim.
Here we provide a different argument of our own that applies to our augmented tree and relies on previous results.

\begin{proof}
Assume that $\alpha$ and $\beta$ are equivalent geodesic rays with vertices $\dots a_2 a_1$ and $\dots b_2 b_1$, respectively.
Using \cref{lem:squares:lift}, it is easy to see that, in general, if an augmented tree $\AT$ has no geodesic $n$-squares then $\dist_\AT(a_i,b_i) < n$ for all $i \geq 1$.
By \cref{prop:IFS:no:2:squares}, up to passing to powers (which we can do without changing the attractor of the IFS nor boundary of $\AT_\Phi$ and the asymptotic equivalence, see \cref{rmk:IFS:power:history:graph,,lem:IFS:power:quasi:isometry}), we can assume that $\AT_\Phi$ has no geodesic $2$-squares.
This means that $\dist_{\AT_\Phi}(a_i,b_i) \leq 1$ for all $i \geq 1$.
If $\dist_{\AT_\Phi}(a_i,b_i)$ is eventually $0$ then there is nothing to prove, so let it be eventually $1$.
By \cref{thm:IFS:history:graph,,rmk:IFS:Crit:PCrit:colors}, this means that $K_{a_i} \cap K{b_i}$ is eventually the same singleton $\{p\} \subseteq K$.
We know that $(\alpha)\chi = \bigcap_{i \geq 1} K_{a_i}$ and $(\beta)\chi = \bigcap_{i \geq 1} K_{b_i}$ are singletons, so they must both be $p$ and, in particular, they must be equal.
\end{proof}

Thus, if $X$ is the limit space of the VERS $\R_\Phi$ (i.e., the boundary of $\AT_\Phi$), the quotient map $\overline{\chi} \colon X \to K$ is well-defined.
Let us show that it is a homeomorphism.

\begin{theorem}
\label{thm:IFS:attractor:is:VERS:limit:space}
The attractor of an injective pcf IFS $\Phi$ is homeomorphic to the limit space of the VERS $\R_\Phi$.
\end{theorem}

\begin{proof}
The map $\overline{\chi}$ is surjective because every point has an address (see \cref{sub:IFS:background}).
As for injectivity, assume that two geodesic rays $\alpha$ and $\beta$ of $\AT_\Phi$ (say with vertices $\dots a_2 a_1$ and $\dots b_2 b_1$, respectively) have the same image via $\chi$.
This means that $\bigcap_{n\geq1} K_{a_n} = \bigcap_{n\geq1} K_{b_n}$ is non-empty, so the intersections $K_{a_n} \cap K_{b_n}$ are non-empty.
By \cref{thm:IFS:history:graph}, this implies that there are edges $a_n \leftrightarrow b_n$, so $\alpha$ and $\beta$ are asymptotically equivalent, as needed.
Since $\chi$ is a bijection between two compact Hausdorff spaces, we now only need to show that it is continuous.

Even if the augmented tree considered in \cite{SSIFS} is different from our history graph, the proof of continuity of \cite[Theorem 4.3]{SSIFS} (which in turn is based on the proof of \cite[Theorem 3.24]{Kaimanovich}) works with minimal modifications.
Indeed, Equation 2.1, Proposition 2.2 and Theorem 2.3 \cite{SSIFS} are about any hyperbolic augmented tree;
as already noted, Lemma 4.1 of \cite{SSIFS} and its proof also work in our setting;
finally, the inequality $\mathrm{diam}(K_{x_n \dots x_1}) \leq C^n \mathrm{diam}(K)$ often used in the proof of \cite[Theorem 4.3]{SSIFS} applies to our case too.
\end{proof}


\section{Belk-Forrest edge replacement systems}
\label{sec:ERSs}

In this section we realize as limit spaces of VERSs the limit spaces of expanding ERSs (edge replacement systems) introduced by Belk and Forrest in \cite{BF19}.

To be consistent with the literature about ERSs, here we will use right-infinite words in place of left-infinite words, and every word will be read from left to right.

\subsection{Background on Belk-Forrest ERSs}

Let us recall the basic notions.

\subsubsection{Edge replacement systems}

\begin{definition}
An \textbf{edge replacement system} (or \textbf{ERS}) consists of:
\begin{itemize}
    \item a set $C$ of colors;
    \item a \textbf{base graph} $\E_0$ colored by $C$;
    \item for each $c \in C$, a \textbf{replacement graph} $X_c$ colored by $C$ and equipped with two vertices $\iota_c$ and $\tau_c$, called the \textit{initial} and \textit{terminal} vertices of $X_c$.
\end{itemize}
\end{definition}

Given a graph $\E$ that is colored by $C$ and a $c$-colored edge $e \in E(\E)$, we can \textbf{expand} $e$ in $\E$ by replacing it with $X_c$, identifying $\iota(e)$ and $\tau(e)$ with $\iota_c$ and $\tau_c$, respectively.
We will often refer to this operation as an ERS expansion, to avoid confusion with VERS expansions.

\subsubsection{The gluing relation}

\begin{definition}
\label{def:full:expansion:sequence:ERS}
Given an ERS, its \textbf{full expansion sequence} is the sequence of graphs $\E_i$ such that $\E_1 = X_0$ and $\E_{i+1}$ is obtained by expanding every edge of $\E_i$.
\end{definition}

The \textbf{symbol space} of an ERS $\ERS$, denoted by $\Omega_\ERS$, is the set of infinite words $e_1 e_2 \dots$ such that, for all $n \in \mathbb{N}$, the prefix $e_1 \dots e_n$ is an edge of $\E_n$.
This is an initial edge shift, and is thus a Stone space when equipped with the topology generated by the cylinders (the sets of elements of $\Omega_\ERS$ with a given prefix).

\begin{definition}
The \textbf{gluing relation} is the relation $\sim$ on $\Omega_\ERS$ defined by
$$x_1 x_2 \dots \sim y_1 y_2 \dots \iff \forall n \in \mathbb{N},\ x_1 \dots x_n \text{ and } y_1 \dots y_n \text{ are equal or adjacent}.$$
\end{definition}

This is not always an equivalence relation, but it is if the ERS satisfies certain conditions, which we recall right below.

\subsubsection{Expanding ERSs and their limit spaces}

\begin{definition}[Definition 1.8 of \cite{BF19}]
\label{def:expanding:ERS}
An ERS is \textbf{expanding} if
\begin{itemize}
    \item the base and replacement graphs do not feature isolated vertices (i.e., vertices on which no edge is incident);
    \item each replacement graph $X_c$ features at least a vertex other than $\iota_c$ and $\tau_c$; 
    \item in each replacement graph $X_c$ there is no edge joining $\iota_c$ with $\tau_c$.
\end{itemize}
\end{definition}

It is proved in \cite[Proposition 1.9]{BF19} that the gluing relation of an expanding ERS is an equivalence relation, which prompts the following definition.

\begin{definition}
The \textbf{limit space} of an expanding ERS $\ERS$ is $\Omega_\ERS / \sim$.
\end{definition}

\begin{example}
\label{ex:ERS:basilica}
Consider the ERS $\ERS_\BasilicaSpc$ with a unique color $\star$ and with base and replacement graphs depicted in \cref{fig:ERS:basilica}.
The limit space is homeomorphic to the basilica Julia set, depicted in \cref{fig:basilica}.
\end{example}

\begin{figure}
\centering
\begin{tikzpicture}
    \begin{scope}[xshift=-4.5cm]
    \node at (-3/4,0) {$X_0 =$};
    \node[vertex] (center) at (1,0){};
    \draw[edge] (center) to[out=150,in=210,min distance=1.2cm] node[left]{$L$} (center);
    \draw[edge] (center) to[out=-30,in=30,min distance=1.2cm] node[right]{$R$} (center);
    \end{scope}
    \begin{scope}
    \node at (-2/3,0) {$X_{\star} =$};
    \node[vertex] (left) at (0,0){};
    \node[vertex] (center) at (1.25,0){};
    \node[vertex] (right) at (2.5,0){};
    \node[above] at (left) {$\iota_{\star}$};
    \node[above] at (right) {$\tau_{\star}$};
    \draw[edge] (left) to node[above]{$0$} (center);
    \draw[edge] (center) to[out=60,in=120,min distance=1cm] node[above]{$1$} (center);
    \draw[edge] (center) to node[above]{$2$} (right);
    \end{scope}
\end{tikzpicture}
\caption{An ERS for the basilica Julia set.}
\label{fig:ERS:basilica}
\end{figure}
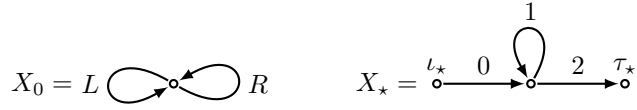

\subsection{VERSs for Belk-Forrest ERSs}

Let us start by defining the tools that we will use to associate every $\E_i$ of the ERS to some $\Gamma_i$ of a VERS.

\begin{definition}
\label{def:graph:partition}
We say that graphs $A_1, \dots, A_k$ form a \textbf{partition} of a graph $A$ if
\begin{itemize}
    \item $V(A_i) \subseteq V(A)$ for each $i$ and $\bigcup_{i=1,\ldots,k} V(A_i) = V(A)$;
    \item $E(A_i) \subseteq E(A)$ for each $i$ and $\bigcup_{i=1,\ldots,k} E(A_i) = E(A)$;
    \item $E(A_i) \cap E(A_j) = \emptyset$ for each $i \neq j$.
\end{itemize}
\end{definition}

\begin{definition}
\label{def:barycentric:subdivision}
Given a graph $\Gamma$ colored by a set $C$, its \textbf{barycentric subdivision} is the graph $\bary{\Gamma}$ colored by the union of two disjoint copies $C_\iota$ and $C_\tau$ of $C$, whose set of vertices is $V(\Gamma) \cup E(\Gamma)$ and whose edges are
\begin{itemize}
    \item $v \rightarrow e$, colored by $c_\iota$, if $v \in V(\Gamma)$, $e \in E(\Gamma)$ is colored by $c$ and $\iota(e)=v$;
    \item $e \rightarrow w$, colored by $c_\tau$, if $w \in V(\Gamma)$, $e \in E(\Gamma)$ is colored by $c$ and $\tau(e)=w$.
\end{itemize}
We equip barycentric subdivisions with a type $\type$ that maps every element of $V(\Gamma)$ to $\vtype$ and each edge $e \in E(\Gamma)$ to its color $\col(e)$, with $T = \{\vtype\} \cup C$ the set of types.
\end{definition}

\begin{remark}
\label{rmk:ERS:kappa}
Let $\kappa$ be the map $C_\iota \cup C_\tau \to \left( C \cup \{\vtype\} \right)^2$ that maps $c_\iota$ to $(\vtype,c)$ and $c_\tau$ to $(c,\vtype)$, for each $c \in C$.
Then the barycentric subdivision of a graph $\Gamma$ colored by $C$ is a $\kappa$-compatible graph (\cref{def:kappa:compatible:colors:types}).
\end{remark}

Given an ERS $\ERS$ with base graph $X_0$ and replacement graphs $X_c$ ($c \in C$), we build a VERS $\R_\ERS$ as follows.
\begin{itemize}
    \item The graph $\Sigma$ that defines the initial edge shift $\V$ is as follows:
    \begin{itemize}
        \item the set of vertices of $\Sigma$ is $C \cup \{\vtype, s\}$;
        \item $\Sigma$ has a loop at $\vtype$, which we call $V$;
        \item for each $c'$-colored $e \in E(X_c)$, $\Sigma$ has an edge $e$ from $c$ to $c'$;
        \item for each $v \in V(X_c) \setminus \{\iota_c,\tau_c\}$, $\Sigma$ has an edge $v$ from $c$ to $\vtype$;
        \item for each $c$-colored $e \in E(X_0)$, $\Sigma$ has an edge $e$ from $s$ to $c$;
        \item for each $v \in E(X_0)$, $\Sigma$ has an edge $v$ from $s$ to $\vtype$.
    \end{itemize}
    \item The set of colors of $\R_\ERS$ is the disjoint union $\{c_0\} \cup C_\iota \cup C_\tau$, with $C_\iota = \{c_\iota \mid c \in C\}$, $C_\tau = \{c_\tau \mid c \in C\}$ two disjoint copies of the set $C$ of colors of $\ERS$.
    \item For each color $c \in C$, the replacement graphs $R_{c_\iota}$ and $R_{c_\tau}$ form an arbitrary partition of the barycentric subdivision of $X_c$ such that $R_{c_\iota}$ does not feature the vertex $\tau_c$ and $R_{c_\tau}$ does not feature the vertex $\iota_c$.
    The graph $R_{c_0}$ is the barycentric subdivision of the base graph $X_0$ of $\ERS$.
    \item The graph $\Gamma_0$ is a single loop colored by $c_0$.
\end{itemize}

Following our \cref{ex:ERS:basilica}, the graph $\Sigma$ that defines the shift $\V$ and the replacement graphs for the VERS $\R_\BasilicaSpc$ are depicted in \cref{fig:VERS:basilica}.
The set of colors is $\{\textcolor{gray}{c_0},\textcolor{blue}{c_\iota},\textcolor{red}{c_\tau}\}$ and the base graph, not depicted, consists of a gray loop.

\begin{figure}
\centering
\begin{tikzpicture}[font=\small]
    \begin{scope}[xshift=-3.75cm,yshift=0cm]
    \node at (-1.75,1/2) {$\Sigma =$};
    \node[state] (s) at (-1,1){$s$};
    \node[state] (v) at (1,1){$\vtype$};
    \node[state] (b) at (0,-.5){$\textcolor{blue}{b}$};
    \draw[edge] (v) to[out=-30,in=30,min distance=1cm] node[right]{$v$} (v);
    \draw[edge] (s) to node[above]{$\nu_0$} (v);
    \draw[edge] (s) to[out=285,in=135] node[left]{$L$} (b);
    \draw[edge] (s) to[out=345,in=105] node[left]{$R$} (b);
    \draw[edge] (b) to node[right]{$\nu_1$} (v);
    \draw[edge] (b) to[out=150,in=210,min distance=1cm] node[left]{$0$} (b);
    \draw[edge] (b) to[out=240,in=300,min distance=1cm] node[below]{$1$} (b);
    \draw[edge] (b) to[out=-30,in=30,min distance=1cm] node[right]{$2$} (b);
    \end{scope}
    \begin{scope}[xshift=3cm,yshift=2cm]
    \node (eps) at (-2.5,0){$\varepsilon$};
    \draw[edge,gray] (eps) to[out=150,in=210,min distance=1cm] (eps);
    \node at (-1.75,0) {$\to$};
    \node (L) at (-1,0) {$L$};
    \node (C) at (.2,0) {$C$};
    \node (R) at (1.4,0) {$R$};
    \draw[edge,blue] (C) to[out=150,in=30] (L);
    \draw[edge,red] (L) to[out=-30,in=210] (C);
    \draw[edge,blue] (C) to[out=-30,in=210] (R);
    \draw[edge,red] (R) to[out=150,in=30] (C);
    \end{scope}
    \begin{scope}[xshift=3cm]
    \node (i) at (-3.5,0){$\mathrm{i}$};
    \node (t) at (-2.5,0){$\mathrm{t}$};
    \draw[edge,blue] (i) to (t);
    \node at (-1.75,0) {$\to$};
    \node (v1i) at (-1,-.5) {$\nu_1\mathrm{i}$};
    \node (0t) at (.2,-.5) {$0\mathrm{t}$};
    \node (vt) at (1.4,-.5) {$\nu\mathrm{t}$};
    \node (2t) at (2.6,-.5) {$2\mathrm{t}$};
    \node (1t) at (1.4,.5) {$1\mathrm{t}$};
    \draw[edge,blue] (v1i) to (0t);
    \draw[edge,red] (0t) to (vt);
    \draw[edge,blue] (vt) to (1t);
    \end{scope}
    \begin{scope}[xshift=3cm,yshift=-2cm]
    \node (i) at (-3.5,0){$\mathrm{i}$};
    \node (t) at (-2.5,0){$\mathrm{t}$};
    \draw[edge,red] (i) to (t);
    \node at (-1.75,0) {$\to$};
    \node (0i) at (-1,-.5) {$0\mathrm{i}$};
    \node (vi) at (.2,-.5) {$\nu\mathrm{i}$};
    \node (2i) at (1.4,-.5) {$2\mathrm{i}$};
    \node (v1t) at (2.6,-.5) {$\nu_1\mathrm{t}$};
    \node (1t) at (.2,.5) {$1\mathrm{t}$};
    \draw[edge,blue] (vi) to (2i);
    \draw[edge,red] (2i) to (v1t);
    \draw[edge,red] (1t) to (vi);
    \end{scope}
\end{tikzpicture}
\caption{The VERS associated to the ERS $\ERS_\BasilicaSpc$.}
\label{fig:VERS:basilica}
\end{figure}

\subsection{Expanding ERSs produce expanding VERS}

Let us explore the concrete connections between an ERS $\ERS$ and the VERS $\R_\ERS$.

\begin{proposition}
\label{prop:ERS:to:VERS}
Given a graph $X$ that is colored by $C$, let $X'$ be its full ERS expansion.
Then the VERS expansion of the barycentric subdivision $\bary{X}$ is the barycentric subdivision of the ERS expansion $\bary{(X')}$.
\end{proposition}

\begin{proof}
By the definition of $\R_\ERS$, the VERS expansion of $\bary{X}$ is obtained by replacing, for each color $c \in C$, every $c_\iota$- and a $c_\tau$-colored edges of $X$ by the two pieces of a partition of the barycentric subdivision of $X_c$.
Thus, replacing the two edges of $\bary{X}$ that correspond to the two halves of an edge of $X$ results in a replacement of the two edges by the barycentric subdivision of $X_c$.
Ultimately, this shows that the VERS expansion of $\bary{X}$ is the barycentric subdivision of $X'$.
\end{proof}

\begin{corollary}
\label{cor:ERS:to:VERS:full:expansion}
The $n$-th expansion $\Gamma_n$ of the VERS $\R_\ERS$ is the barycentric subdivision of the $n$-th graph $\E_n$ of the full expansion sequence of the ERS $\ERS$, $\forall n \geq 1$.
\end{corollary}

\begin{proof}
For $n=1$, the graph $\Gamma_1$ is obtained by expanding the graph $\Gamma_0$, which consists of a single loop colored by $c_0$.
This produced a graph $R_{c_0}$ that is the barycentric subdivision of the base graph $X_0$ of $\ERS$, which is the first term $\E_1$ of the full expansion sequence.
Note also that each edge of $\Gamma_1$ is colored by $c_\iota$ if it consists of the initial half of an edge of $\E_1$ and by $c_\tau$ if it consists of the final half of an edge of $\E_1$.
Recall that $\E_{n+1}$ is obtained by expanding every edge of $\E_n$ (\cref{def:full:expansion:sequence:ERS}).
We can thus conclude by induction:
assuming that $\Gamma_n = \bary{(\E_n)}$, by \cref{prop:ERS:to:VERS} we have that $\Gamma_{n+1} = \bary{(\E_{n+1})}$.
\end{proof}

\begin{theorem}
\label{thm:expanding:ERS:is:expanding:VERS}
If an ERS $\ERS$ is expanding (in the sense of \cref{def:expanding:ERS}) then its VERS $\R_\ERS$ is expanding (in the sense of \cref{def:expanding:VERS}).
\end{theorem}

\begin{proof}
Let $\ERS$ be an expanding ERS and consider the VERS $\R_\ERS$.
Assume that $\Gamma$ is a $\kappa$-compatible path of length $4$, say that it is $v_0 \xrightarrow{e_1} \cdots \xrightarrow{e_4} v_4$.
Since it is $\kappa$-compatible, the vertices $v_0, \dots, v_4$ alternate between the type $\vtype$ and types $c \in C$, so let us distinguish between two possibilities.

Assume that $v_0$ has type $\vtype$, so $v_2$ and $v_4$ have type $\vtype$ too whereas $v_1, v_3$ have types $c_1, c_2 \in C$, respectively.
Then $\Gamma$ is the barycentric subdivision of a path $\tilde\Gamma$ of length $2$, whose two edges have colors $c_1$ and $c_2$.
Consider the ERS expansion of $\tilde\Gamma$:
since $\ERS$ is expanding, the ERS expansions of either of the two edges of $\tilde\Gamma$ either disconnects its initial and terminal vertices or strictly increases their distance.
Hence, if one of the edges of $\tilde\Gamma$, when expanded, has its initial and terminal vertices disconnected, then its full expansion cannot feature a geodesic between $v_0$ and $v_4$.
If instead none of the edges, when expanded, have their initial and terminal vertices disconnected, then the length of $\tilde\Gamma$ at least doubles.
Either way, a geodesic joining the endpoints of the path $\tilde\Gamma$ cannot have length $2$.
By \cref{prop:ERS:to:VERS}, the VERS expasion of $\Gamma$ is the barycentric subdivision of the ERS expansion of $\tilde\Gamma$, so the expansion of $\Gamma$ cannot feature geodesics of length less than $5$ joining successors of its endpoints.

Assume instead that $v_0$ has type $c_1 \in C$, so $v_2, v_4$ have types $c_2, c_3 \in C$, respectively, whereas $v_1$ and $v_3$ have type $\vtype$.
In this case, we cannot argue that $\Gamma$ is the barycentric subdivision of a graph.
However, the subgraph $v_1 \to v_2 \to v_3$ is the barycentric subdivision of a graph $\tilde{\Gamma}$ consisting of single edge $e$.
As in the previous paragraph, the ERS expansion of $e$ strictly increases the distance between the two vertices or disconnects them.
Then the same holds for the graph $\Gamma$:
its VERS expansion either sees $v_1$ and $v_3$ disconnected or their distance increases.
Using \cref{prop:ERS:to:VERS}, since the ERS is expanding and by construction of the VERS $\R_\ERS$, there is at least one edge in the VERS expansion of $\Gamma$ that is incident on $v_1V$ and on $v_0x$ for some $x \in E(X_{c_1})$ and one that is incident on $v_3V$ and on $v_4y$ for some $y \in E(X_{c_3})$.
The vertices $v_0x$ and $v_3y$ of the VERS expansion of $\Gamma$ cannot be joined by any path that does not involve the expansion of $e$, so ultimately any path between them has length at least $6$, as needed.
\end{proof}

The converse of \cref{thm:expanding:ERS:is:expanding:VERS} does not hold, since the notion of expansivity of ERSs requires each edge replacement to immediately ``expand'', instead of allowing for multiple steps.
For example, consider an ERS $\ERS$ with colors blue and red that replaces each red edge with a blue edge and each blue edge with a path of two red edges.
It is not expanding according to \cref{def:expanding:ERS} (it is, in a certain sense, ``expanding in $2$ steps''), but the VERS $\R_\ERS$ is expanding.

\subsection{The limit space of an ERS is the Gromov boundary of a VERS}

For the next proofs, recall \cref{def:asymptotic:relation,,def:VERS:limit:space}.

\begin{proposition}
\label{prop:gluing:implies:asymptotic}
Fix an ERS $\ERS$ and the associated VERS $\R_\ERS$.
If $\alpha, \beta \in \Omega$ are equivalent under the gluing relation of $\ERS$, then they are equivalent under the asymptotic relation of $\R_\ERS$.
\end{proposition}

\begin{proof}
If $\alpha$ and $\beta$ are equivalent under the gluing relation, each of their prefixes of equal length $i$ are the same or they are adjacent in $\E_i$.
By \cref{cor:ERS:to:VERS:full:expansion}, these same prefixes are vertices at distance at most $2$ in $\Gamma_i$ for all $i$, so $\alpha$ and $\beta$ are in the same asymptotic class.
\end{proof}

\begin{lemma}
\label{lem:asymptotic:classes:prefix:distance}
Assume that $\ERS$ is an expanding ERS and let $\alpha = a_1 a_2 \ldots$ and $\beta = b_1 b_2 \ldots$ be asymptotically equivalent elements of $\V$.
Then there exists $N \in \{0,1,2\}$ such that $\dist_{\Gamma_k}(a_1 \ldots a_k, b_1 \ldots b_k) = N$ for all $k$ large enough.
\end{lemma}

\begin{proof}
The sequence $\dist_{\Gamma_k}(a_1 \ldots a_k, b_1 \ldots b_k)$ is non-decreasing by \cref{rmk:lift:distance}.
Hence, it suffices to show that, if there exists $k$ such that $\dist_{\Gamma_k}(a_1 \ldots a_k, b_1 \ldots b_k) \geq 3$, then $\alpha$ and $\beta$ are not asymptotically equivalent.
Assuming that $\dist_{\Gamma_k}(a_1 \ldots a_k, b_1 \ldots b_k) \geq 3$, there is a path of length at least $3$ joining $a_1 \ldots a_k$ and $b_1 \ldots b_k$ in $\Gamma_k$.
Since $\Gamma_k$ is the barycentric subdivision of $\E_k$ by \cref{cor:ERS:to:VERS:full:expansion}, the path must include a subpath $v \leftrightarrow e \leftrightarrow w \leftrightarrow f$, where $v,w$ represent vertices of $\E_k$ and $e,f$ represent edges.
Since $\ERS$ is an expanding ERS, the expansion of $e$ produces a strictly longer path between $v$ and $w$ or no path at all joining them.
This applies to any path between $a_1 \ldots a_k$ and $b_1 \ldots b_k$, so $\dist_{\Gamma_{k+1}}(a_1 \ldots a_{k+1}, b_1 \ldots b_{k+1}) > \dist_{\Gamma_k}(a_1 \ldots a_k, b_1 \ldots b_k)$, showing that $\alpha$ and $\beta$ are not asymptotically equivalent. 
\end{proof}

For the following proofs, let us keep in mind that every element of $\V$ either ends with $\overline{V}$ (so its prefixes eventually represent the same vertex) or does not feature the symbol $V$ at all (in which case all of its prefixes represent edges).

\begin{lemma}
\label{lem:ERS:non:trivial:asymptotic:classes}
If $\ERS$ is an expanding ERS, then each non-trivial asymptotic class includes some $\nu \in \V$ that ends with $\overline{V}$ and some $\epsilon \in \V$ that does not feature $V$.
\end{lemma}

\begin{proof}
Let $\xi$ be a non-trivial asymptotic class.
Let us first assume that it includes a $\nu$ that ends with $\overline{V}$.
Recall that the ERS $\ERS$ features no isolated vertices, because it is expanding (\cref{def:expanding:ERS}).
Then every $\E_k$ (for all $k$ large enough) features edges that are incident on the vertex represented by the $k$-th prefix of $\nu$.
Each such edge of $\E_{k+1}$ comes from an expansion of an edge in $\E_k$ with the same property (that of being incident on a prefix of $\nu$), so it is easy to build an element of $\V$ whose every long enough prefix is an edge that is incident on the prefix of $\nu$ with the same length.
Such element of $\V$ never features the symbol $V$ and each prefix is at distance $1$ with the prefix of $\nu$, so it is asymptotically equivalent to $\nu$, as needed.\

Assume now that $\xi$ includes an $\alpha$ that does not feature the symbol $V$.
Since $\xi$ is non-trivial, it features some element $\beta$ other than $\alpha$.
If $\beta$ features the symbol $V$ then it must end with $\overline{V}$ and there is nothing to prove, so let us assume that $\beta$ does not feature the symbol $V$.
By \cref{lem:asymptotic:classes:prefix:distance}, every long enough prefix of $\alpha$ and $\beta$ is at distance $2$ (the distance cannot always be $0$ or $\alpha$ and $\beta$ would be the same, and it cannot be odd as the prefixes represent edges in barycentric subdivisions).
In terms of graphs of the ERS, if $\alpha=a_1 a_2 \dots$ and $\beta=b_1 b_2 \dots$, this means that $a_1 \dots a_k$ and $b_1 \dots b_k$ are adjacent edges of $\E_k$ for all $k$ large enough, by \cref{cor:ERS:to:VERS:full:expansion}.
Note that, since $\ERS$ is expanding, if one assumes that $a_1 \dots a_k$ and $b_1 \dots b_k$ are both incident on $x_1 \dots x_m V^{k-m}$ and that $a_1 \dots a_k a_{k+1}$ and $b_1 \dots b_k b_{k+1}$ are adjacent, then $a_1 \dots a_k a_{k+1}$ and $b_1 \dots b_k b_{k+1}$ must be incident on the representative $x_1 \dots x_m V^{k+1-m}$ for the same vertex.
In particular, this means that there is some $\nu = x_1 \dots x_m \overline{V}$ whose long enough prefixes have distance $1$ from those of $\alpha$ and $\beta$, so $\nu \in \xi$ as needed.
\end{proof}

\begin{proposition}
\label{prop:asymptotic:implies:gluing}
If $\ERS$ is an expanding ERS, for each non-trivial asymptotic class $\xi$ there exists a unique $w \in \Lang\V$ that does not feature $V$ and such that
$$\xi = \{ w\overline{V} \} \cup \{ w x_1 x_2 \ldots \mid \forall k \in \mathbb{N},\ \dist(w x_1 \ldots x_k, wV^k) = 1 \}.$$
\end{proposition}

\begin{proof}
Assume that $\xi$ is a non-trivial asymptotic class.
By \cref{lem:ERS:non:trivial:asymptotic:classes}, it must feature a $\nu$ that ends with $\overline{V}$ and an $\epsilon$ that does not feature the symbol $V$, say $\nu=x_1 \dots x_k \overline{V}$ (with $x_k \neq V$) and $\epsilon=e_1 e_2 \dots$ (with $e_i \neq V$ for all $i$).

Let us first show that $\nu$ is the only element of $\xi$ that ends with $\overline{V}$.
Assume by contradiction there exists some $\nu' = y_1 \ldots y_m \overline{V} \in \xi$ other than $\nu$.
If $M = \mathrm{max}\{k,m\}$, then the prefixes of length $M$ of $\nu$ and $\nu'$ represent vertices of $\E_M$ that are joined by some edge $e$, by \cref{cor:ERS:to:VERS:full:expansion,,lem:asymptotic:classes:prefix:distance}.
Since the ERS is expanding, as one expands the edge $e$, the distance between $\nu$ and $\nu'$ must increase (possibly becoming infinite if $\nu$ and $\nu'$ end in different connected components).
This contradicts the fact that the distance between the prefixes of $\nu$ and $\nu'$ is bounded, so $\nu'$ cannot belong to the same asymptotic class of $\nu$.
Hence, $\nu$ must be the sole element of $\xi$ ending with $\overline{V}$.

In order to conclude the proof, by \cref{cor:ERS:to:VERS:full:expansion} it only remains to note that the elements of $\xi$ have long enough prefixes at distance at most $1$ from those of $\nu$.
This holds because of \cref{lem:asymptotic:classes:prefix:distance} together with the previously proved fact that $\xi$ does not feature elements ending with $\overline{V}$ other than $\nu$ and the observation that the prefixes of elements that do not feature the symbol $V$ can only have odd distance with the prefixes of $\nu$ by \cref{cor:ERS:to:VERS:full:expansion}.
\end{proof}

\begin{theorem}
\label{thm:ERS:limit:space:is:VERS:limit:space}
The limit space of an expanding ERS $\ERS$ is homeomorphic to thelimit space of the VERS $\R_\ERS$.
\end{theorem}

\begin{proof}
Let $\Omega$ and $X$ be the symbol space and the limit space of the ERS $\ERS$, respectively, and let $\pi_\Omega \colon \Omega \to X$ be the quotient map for the gluing relation.
Let $Y$ be the Gromov boundary of the history graph $\AT_{\R_\ERS}$ for the VERS $\R_\ERS$ and let $\pi_\AT \colon \V \to Y$ be the quotient map for the asymptotic relation.
Consider the injective map $\Phi \colon \Omega \to \V$ that sends each $e_1 e_2 \ldots \in \Omega$ to $e_1 e_2 \ldots \in \V$.
It is continuous, since the topology of $\V$ is generated by the set of cylinders and the preimage of a cylinder of $\V$ is a clopen subset of $\Omega$.

Let $\tilde{\Phi} \colon X \to Y$ be defined by setting $\tilde{\Phi}(\pi_\Omega(\alpha)) = \pi_\AT(\alpha)$ for all $\alpha \in \Omega$.
This really defines a map, since $\pi_\Omega(\alpha)=\pi_\Omega(\beta)$ implies $\pi_\AT(\alpha)=\pi_\AT(\beta)$ by \cref{prop:gluing:implies:asymptotic}.
Since $\Phi$ is continuous, passing to the quotient we have that $\tilde{\Phi}$ is continuous too.
It now suffices to show that $\tilde{\Phi}$ is bijective, since $X$ and $Y$ are compact and Hausdorff.

Given any $\alpha \in \V$ let us show that $\pi_\AT(\alpha) \in \tilde{\Phi}(X)$.
If all the digits of $\alpha$ are edges of the base and replacement graphs $X_0$ and $X_c$ ($c \in C$), then $\alpha \in \Omega$ and so $\tilde{\Phi}\left(\pi_\Omega(\alpha)\right) = \pi_\AT(\alpha)$.
If instead $\alpha$ features a digit $V$, then $\alpha = a \overline{V}$ for some finite (possibly empty) prefix $a$ whose digits are all edges of $X_0$ and $X_c$ ($c \in C$).
In this case, consider any $a \beta \in \Omega$ whose digits are all edges of $X_0$ and $X_c$ ($c \in C$) and such that, for all $k \in \mathbb{N}$ larger than the length of $a$, the $k$-th prefix of $a \beta$ is adjacent to the vertex corresponding to $a V$ in the ERS.
Now, all the prefixes of $a \overline{V}$ and $a \beta$ of equal lengths have distance at most $1$, so $\pi_\AT(a \overline{V})=\pi_\AT(a \beta)$.
And $a\beta$ is as in the previous case, so $\pi_\AT(a \overline{V}) = \pi_\AT(a \beta) = \tilde{\Phi}\left(\pi_\Omega(a\beta\right)$.
Hence, $\tilde{\Phi}$ is surjective.
Injectivity follows from \cref{prop:asymptotic:implies:gluing}, so we are done.
\end{proof}

As an consequence, we recover \cite[Theorem 1.25]{BF19}:
limit spaces of expanding ERSs are compact and metrizable.

\begin{remark}
The replacement systems of \cite{ESS} deal with directed hypergraphs:
instead of edges, with two boundary vertices $\iota(e)$ and $\tau(e)$, one has \textit{hyper}edges, which have any finite amount of boundary vertices $\lambda_1(e), \dots, \lambda_k(e)$.
One may consider the following notion of barycentric subdivision of hypergraphs:
$V(\bary{\Gamma}) = V(\Gamma) \cup E(\Gamma)$ (as for graphs) and each hyperedge $e$ of $\Gamma$ with boundary vertices $\lambda_1, \dots, \lambda_k$ is replaced by the hyperedges $e_1,\dots,e_k$ of $\bary{\Gamma}$, where each $e_i$ has boundary vertices $\lambda_1', \dots, \lambda_k'$, with $\lambda_j' = \lambda_j$ for all $j \neq i$ and $\lambda_i' = e$.
It looks like the arguments from this section can be generalized in this direction, but doing so would require developing a notion of vertex and \textit{hyper}edge replacement systems, which is beyond the scope of this paper and would significantly extend its length.

Moreover, \cite{ESS} introduces \textbf{\textit{almost} expanding} (hyper)edge replacement systems, where colors are allowed to not expand when the edges they color are isolated (i.e., they never share vertices with other edges).
This allows for limit spaces with isolated points.
The VERS associated to an almost expanding edge replacement system may not be expanding in the sense of \cref{def:expanding:VERS}, but it can be proved that the history graph is hyperbolic because the ``problematic'' geodesics (that generate big squares) never occur in the actual VERS expansions.
This is the same phenomenon discussed in \cref{rmk:expanding:not:iff}.
\end{remark}


\section*{Acknowledgements}

The authors are thankful to Alex Bishop and Alexander Teplyaev for useful discussions and to Rachel Skipper and the University of Utah for their hospitality.

\printbibliography[heading=bibintoc]

\end{document}